\documentclass[final,hidelinks,onefignum,onetabnum]{siamart220329}

\usepackage{lipsum}
\usepackage{amsfonts}
\usepackage{graphicx}
\usepackage{epstopdf}
\usepackage{algorithmic}
\ifpdf
  \DeclareGraphicsExtensions{.eps,.pdf,.png,.jpg}
\else
  \DeclareGraphicsExtensions{.eps}
\fi

\usepackage{microtype}
\usepackage{subcaption}
\usepackage{booktabs}
\usepackage{hyperref}
\usepackage{amsfonts,mathrsfs}
\usepackage{verbatim}
\usepackage{acronym,color}
\usepackage{pifont}
\usepackage[normalem]{ulem}
\usepackage{yhmath}
\usepackage{amsmath}
\usepackage{amssymb}
\usepackage{mathtools}

\newsiamremark{remark}{Remark}
\newsiamremark{hypothesis}{Hypothesis}
\crefname{hypothesis}{Hypothesis}{Hypotheses}
\newsiamthm{claim}{Claim}

\newsiamremark{assumption}{Assumption}

\def\la{\langle}
\def\ra{\rangle}

\def\cF{{\cal F}}
\def\cJ{{\cal J}}

\def\1{{\bf 1}}
\def\A{\mathscr{A}}

\def\ol{\bar}

\def\b{\beta}

\def\a{\alpha}

\def\o{\omega}

\def\R{\mathbb{R}}

\def\re{\mathbb{R}}

\def\dist{{\rm dist}}

\newcommand{\EXP}[1]{\mathsf{E}\!\left[#1\right] }

\def\argmin{\mathop{\rm argmin}}
\def\be{\begin{equation}}
\def\ee{\end{equation}}

\def\bit{\begin{itemize}}
\def\eit{\end{itemize}}

\newcommand{\cmark}{\ding{51}}
\newcommand{\xmark}{\ding{55}}
\newcommand{\remove}[1]{}

\headers{Randomized Feasibility-Update Algorithms for VIs}{A. Chakraborty and A. Nedi\'c}

\title{Randomized Feasibility-Update Algorithms for Variational Inequality Problems}

\author{Abhishek Chakraborty\thanks{School of Electrical Computer and Energy Engineering, Arizona State University, Tempe, Arizona, USA
  (\email{achakr61@asu.edu}, \email{Angelia.Nedich@asu.edu}). \funding{This work was funded by the ONR award N00014-21-1-2242 and the NSF grant CIF-2134256.}}
\and Angelia Nedi\'c\footnotemark[2]}

\usepackage{amsopn}

\ifpdf
\hypersetup{
  pdftitle={Randomized Feasibility-Update Algorithms for Variational Inequality Problems},
  pdfauthor={A. Chakraborty and A. Nedi\'c}
}
\fi

\begin{document}

\maketitle

\begin{abstract}
This paper considers a variational inequality (VI) problem arising from a game among multiple agents, where each agent aims to minimize its own cost function subject to its constrained set represented as the intersection of a (possibly infinite) number of convex functional level sets. A direct projection-based approach or Lagrangian-based techniques for such a problem can be computationally expensive if not impossible to implement. To deal with the problem, we consider randomized methods that avoid the projection step on the whole constraint set by employing random feasibility updates. In particular, we propose and analyze such random methods for solving VIs based on the projection method, Korpelevich method, and Popov method. We establish the almost sure convergence of the methods and, also, provide their convergence rate guarantees. We illustrate the performance of the methods in simulations for two-agent games.
\end{abstract}





\section{Introduction}\label{sec:Intro}
Multi-agent games arise in a wide range of domains such as economics~\cite{van2012dynamic}, transportation science~\cite{acciaio2021cournot}, multi-agent reinforcement learning~\cite{lanctot2017unified}, and many other areas. A multi-agent game is characterized by an agent specific payoff function representing the cost incurred by an agent's decision and the decisions taken by all the other agents playing the game. An agent can minimize the cost with respect to its own decision variable subject to the agent's constraints~\cite{facchinei2003finite}. With a well defined mapping, corresponding to the gradients of the payoff functions of all the players, a multi-agent game can be formulated as a Variational Inequality (VI) problem, which can be solved with standard algorithms  \cite{solodov1999new, korpelevich1976extragradient, popov1980modification, tseng1995linear} under some assumptions on the mapping of the game. These algorithms typically employ the projection on the players' constraint sets.
This paper deals with the case when the constraint sets are specified as {\it the intersection of a large number of convex functions' level sets},
which prohibit the direct use of the projection operation. 
To deal with such problems, this paper proposes the methods with random feasibility updates, which have been studied for optimization problems, but remained less explored for VIs.
%
%

Related work in the optimization domain dealing with large number of functional constraints
dates back to~\cite{polyak1969minimization, polyak2001random}, where a feasibilty problem was considered. Subsequently, randomized methods for a convex feasibility problem were studied in~\cite{nedic2010random, necoara2021minibatch}. 
The papers~\cite{liu2015averaging,wang2016stochastic} deal with finitely many constraints, where a direct projection on a randomly selected set is used.
The paper \cite{fercoq2019almost} deals with the optimization problem with infinitely many constraints using duality and smoothing techniques, whereas~\cite{nedic2011random, nedic2019random, necoara2022stochastic} use random feasibility updates
and solve the optimization problems with infinitely, finitely and infinitely many constraints, respectively.
Another approach to construct a  projection free method is the use of conditional gradients/Frank-Wolfe techniques~\cite{woodstock2023splitting} and level conditional gradient method \cite{cheng2022functional}, which are currently limited to problems with finitely many constraints.



\begin{table*}[t]\centering
\begin{tabular}{ |c|c|c|c|c| } 
\hline
Work & Method & Mapping & \# Constraints & Project on $X_{j}$ \\
\hline
\cite{wang2015incremental} & Proj & Strongly monotone & Finite & \cmark\\
\hline
\cite{cui2021analysis} & SPRG, SSE & Monotone & Finite & \cmark\\
\hline
Ours & Proj, Kor, Popov &  Strongly monotone & Infinite & \xmark\\
\hline
\end{tabular}
\caption{Randomized projection schemes for VIs. 
with Lipschitz continuous mappings. ``SPRG" and ``SSE" refer to Stochastic Projected Reflected Gradient and Stochastic Subgradient Extragradient methods, respectively, of~\cite{cui2021analysis}. ``Proj" and ``Kor" refer to the Projection and the Korpelevich methods, respectively; note that Korpelevich method is the same as the extra-gradient method studied in \cite{cui2021analysis}, whereas SPRG is a version of Malitsky's method \cite{malitsky2015projected}. The last column indicates whether or not the projection on a randomly selected set $X_{j}$ is used.
Our methods do not use such projections and can handle infinitely many constraints.} 
\label{tab:previous_works_VI}
\end{table*}


In the context of VIs, the paper~\cite{boroun2023projection} deals with nonconvex-concave saddle point problem and uses primal-dual conditional gradient method. The paper~\cite{alizadeh2023randomized} solves stochastic monotone Nash games using randomized block primal-dual scheme. Both these works are for settings with finitely many constraints. Moreover, the use of primal-dual techniques requires updates of many dual variables at each iteration, which can make the computation cumbersome and  impossible to implement
for infinitely many constraints. Works~\cite{chavdarova2023primal,yang2022solving} use log barrier methods for constraint satisfaction, whereas~\cite{zhang2024primal} develops a primal scheme without the need to know the optimal Lagrange multiplier to solve a generalized VI problem. The first work on incremental and random projection algorithms for solving VIs is~\cite{wang2015incremental}, where a stochastic VI with a strongly monotone and Lipschitz continuous
expected-value mapping is considered. Subsequently,~\cite{cui2021analysis} extended the work in~\cite{wang2015incremental} to stochastic VIs for a monotone, weakly sharp and Lipschitz continuous mapping. None of these works have developed results for the case of infinitely many constraints. Also, they have not considered functional constraints. This paper addresses these gaps for deterministic VIs (see Table \ref{tab:previous_works_VI}).

The motivation for considering the problems with infinitely many constraints such as stochastic constraints of the form $\{x\in\re^{n}\mid g(x,\xi)\le 0 \}$, where $\xi$ is a continuous random variable, comes from problems emerging in online settings dealing with multi-agent reinforcement learning with safety constraints \cite{gu2021multi, miryoosefi2019reinforcement}, or energy constraints \cite{bai2024learning}, and also in the context of imitation learning with exploration~\cite{shao2024imitating}. 
Problems with many constraints can also appear in Nash-Cournot equilibrium problems in economics, constrained GANs \cite{heim2019constrained}, and in almost sure chance constrained problems~\cite{fercoq2019almost}.

\noindent
\textit{Contributions:} (1) To the best of our knowledge, this is the first work that considers VIs with potentially infinitely many constraints. Also, it is the first to consider randomized methods for VIs with functional constraints.  
A modified versions of the projection method, Korpelevich method, and Popov method are proposed and analyzed. They deal with (potentially infinitely many) functional constraints by utilizing
random feasibility updates~\cite{polyak2001random, nedic2019random, gower2022cutting}. 
(2)~For all the proposed algorithms, the convergence of the iterates to the solution of the VI is established in almost sure sense and in expectation. 
We do not assume that the constraint set is compact.
(3)~Per iterate convergence rate is shown for the expected squared-distance of the iterates to the solution of the VI problem.
(4)~For the first time,  per iterate convergence to the constraint set is shown, in expectation, 
which has a geometric decay with respect to the number of random constraints that are drawn
times $O(1/\sqrt{k})$. This is unlike~\cite{cui2021analysis}, which shows the convergence rate of only $O(1/\sqrt{k})$ for weighted averaged iterates. It is also unlike the convergence result of \cite{wang2015incremental}, which is not per iterate but rather for the minimum distance achieved at a given iteration and, also, provides a rate for approaching a ``vicinity of the set". 

While the random feasibility updates are well understood for the feasibility \cite{necoara2021minibatch} and optimization problems \cite{nedic2019random}, they were not studied in the context of VIs. The main challenge here compared to optimization is in the analysis for the VIs. In particular, the analysis for the proposed modified Projection, Korpelevich and Popov methods is intricate, where careful splitting of the error terms allows us to obtain a number of random samples for the feasibility updates independent of the problem parameters. Consequently, the methods work even with a single sampled constrained set at a time.

\noindent
\text{\it Paper organization:}
We state the problem and its VI formulation in Section~\ref{sec:problem}. The random feasibility update algorithm and its related proofs are in Section~\ref{sec:inf_updates}. The modified Projection method, Korpelevich method, and Popov method that utilize random feasibility updates are proposed and analyzed (with their related convergence guarantees) in 
Sections~\ref{sec:proj_method}, \ref{sec:Korpelevich_method}, and \ref{sec:Popov_method}, respectively. 
Section~\ref{sec:simulations} shows the simulation results on two toy problems, a matrix game and a game of imitation learning with exploration. Finally, some conclusions, limitations, and future directions are provided in Section~\ref{sec:conclusion}.

\noindent
\textit{Notations:} We consider the vector space $\R^n$.  
 The inner product of two vectors $x$ and $y$ is denoted by $\la x,y\ra$, 
 while $\|\cdot\|$ is the standard Euclidean norm. The distance of a vector $\ol x$ 
 from a closed convex set $X$ is given by
 $\dist(\ol x,X)=\min_{x\in X}\|x-\ol x\|.$
 The projection of a vector $\ol x$
 on the set $X$ is  $\Pi_X[\ol x]=\argmin_{x\in X}\|x-\ol x\|^2.$
 For a scalar $a$, we use $a^+$ to denote the maximum of $a$ and~$0$,
 i.e.\ $a^+=\max\{a,0\}$.
 We use $\EXP{\o}$ for the expectation of a random variable~$\o$.
 We often abbreviate {\it almost surely} by {\it a.s.}


\section{Problem Formulation and Preliminaries}\label{sec:problem}
We consider a game among $J$ agents which are indexed by $j \in \cJ = \{1,\ldots,J\}$.
Each agent $j$ has a cost function $\theta_j(x_j,x_{-j}): \R^n \rightarrow \R$ that depends on its decision $x_j \in \R^{n_j}$ 
and the decision of all the other agents, denoted by $x_{-j}$ (the concatenated decisions of all the other agents).
Every agent $j$ wants to minimize its own cost function subject to a constrain set: 
\begin{align}
    \min_{x_j \in \R^{n_j}}\quad &\theta_j(x_j,x_{-j}) \nonumber\\
    \text{s.t.} \quad &x_j \in S_j = X_j \cap Y_j \quad \text{with}
     \quad X_j=\left(\cap_{a_j\in \A_j} X_{a_j} \right), \nonumber\\
     \text{and} \quad &X_{a_j} = \{x_j\in\re^{n_j}\mid g_{a_j}(x_j)\le 0\}, \label{eq-problem}
\end{align}
where $\A_j$ is an index set (possibly infinite) for each $j$. 
In this paper, we view $\A_j$ as an index set and $a_j$ as its associated index element. One can also view the index $a_j$ as a random variable taking values in the set $\A_j$. In this case, given a random value $a_j$, 
the constraint set $X_{a_j}$ would coincide with $\{x_j \in \R^{n_j} \mid g(x_j,a_j) \leq 0\}$, similar to \cite{necoara2022stochastic}. 

We let $n=\sum_{j=1}^J n_j$ and $x= (x_1, \ldots, x_J)^T \in \R^n$. It is assumed that the set $Y_j$ is simple to project on, but the projection on the set $X_j\cap Y_j$ is computationally demanding, and we seek to reduce the computation by employing random projections. We let the sets $X$, $Y$, and $S$ be the Cartesian products of the sets $X_j$, $Y_j$, and $S_j$, respectively, for all $j$, i.e., $X= \prod_{j=1}^J X_j$, $Y= \prod_{j=1}^J Y_j$, and $S= \prod_{j=1}^J S_j$.
Next, we present an assumption on the sets $Y_j$ and $X_j$.
\begin{assumption}\label{asum_closed_set}
For all $j \in \cJ$, 
the set $Y_j\subseteq\re^{n_j}$ is closed and convex, the set $S_j = X_j \cap Y_j$ is  non-empty, and the function $g_{a_j}:\re^{n_j}\to\re$ is convex for all $a_j \in\A_j$. 
\end{assumption}
By Assumption~\ref{asum_closed_set}, each function $g_{a_j}$ is continuous over $\re^{n_j}$ \cite[Proposition 1.4.6]{bertsekas2003convex}, implying that 
 the set $\{x_j\in\re^{n_j}\mid g_{a_j}(x_j)\le0\}$ is closed and convex. Hence, the set $X_j$ is closed and convex and so is the Cartesian product set $X$. Additionally, for each $j \in \cJ$, the subdifferential set $\partial g_{a_j}(x_j)$ is nonempty for all $x_j\in \re^{n_j}$ and all $a_j \in\A_j$~\cite[Proposition 4.2.1]{bertsekas2003convex}, implying that $\partial g_{a_j}^+(x_j)\ne\emptyset$
for all $x_j\in \re^{n_j}$ and $a_j \in\A_j$. The next assumption is on the functional constraints and their subgradients.
 
 \begin{assumption}\label{asum-subgrad-bound} 
 For all $j \in \cJ$ and all $a_j\in \A_j$, the constraint function $g_{a_j} : \R^n \rightarrow \R$ has bounded subgradients on the set $Y_j$, 
 i.e., there is a scalar $M_g>0$, such that 
  \begin{align*}
      \|d\| \leq M_g \;, \quad  \forall d \in \partial g_{a_j}(x_j) \;, x_j \in Y_j, a_j\in\A_j, \text{ and } j=1,\ldots,J . 
  \end{align*}
 \end{assumption}
 A bounded $Y_j$ implies bounded subgradients on the set $Y_j$ \cite[Proposition 4.2.3]{bertsekas2003convex}. 

With a differentiable cost function, the game defined in \eqref{eq-problem} can be equivalently written as a VI problem with mapping $F:Y \rightarrow \R^{n}$ defined as the stacked gradient vector of the costs of all the players, i.e. 
\begin{align}
    F(x) = \begin{bmatrix}
\nabla_{x_1} \theta_1 (x), \ldots, \nabla_{x_J} \theta_J (x) 
\end{bmatrix}^T . \nonumber
\end{align}
%
%
%
From now onwards, we \textit{concentrate on solving a variational inequality problem} of finding $x^* \in \text{SOL}(S, F)$ such that 
\begin{align}
    \langle F(x^*), x - x^* \rangle \geq 0 \: , \qquad \forall x \in S , \label{VI_problem}
\end{align}
where $\text{SOL}(S, F)$ is the solution set of VI$(S, F)$. The algorithms we consider will have feasible updates with respect to $Y$, but possibly infeasible updates with respect to the set $X$. We assume that the mapping $F$ is defined over the set $Y$, with additional properties, as follows.
\begin{assumption}\label{asum_lipschitz}
    The mapping $F:Y \rightarrow \R^{n}$ is Lipschitz continuous over $Y$ with the constant $L>0$,
   i.e.,
   \begin{align}
       \|F(x)-F(y)\| \le L \|x-y\| \quad \forall x,y\in Y. \nonumber
   \end{align}
\end{assumption}
\begin{assumption}\label{asum_strong_monotone}
    The mapping $F:Y \rightarrow \R^{n}$ is strongly monotone
   over $Y$ with the constant $\mu>0$, i.e.,
   \begin{align}
       \la F(x) - F(y), x - y \ra \ge \mu \|x-y\|^2 \quad \forall x,y\in Y. \nonumber
   \end{align}
\end{assumption}
Assumption \ref{asum_strong_monotone} implies that the VI$(S,F)$ has a unique solution $x^*\in S$ \cite[Theorem 2.3.3 and Proposition 2.2.3]{facchinei2003finite}. For the mapping $F$, we define the condition number $\kappa$ as
\begin{align}\label{eq-cnum}
    \kappa =\frac{L}{\mu} .
\end{align}

Next, we present two auxiliary results related to near supermartingale convergence, which we use in our analysis of the methods later on.
%
%
%
\begin{lemma}\label{lem_near_super_martingale}
    Let  $\{X_k\}$ be a sequence of non-negative (scalar) random variables, with $\EXP{X_0} < \infty$, and satisfying the following relation 
    almost surely for all $k \geq 0$,
    \begin{align*}
        \EXP{X_{k+1} \mid X_0,\ldots,X_k} &\leq (1- \Tilde{\gamma}_k) X_k + b_k,
    \end{align*}
    where $\{\tilde \gamma_k\}$ and $\{b_k\}$ are deterministic sequences satisfying $0 \leq \Tilde{\gamma}_k \leq 1$, $\sum_{k=0}^\infty \Tilde{\gamma}_k = \infty$, and $b_k \geq 0$ with $\sum_{k=0}^\infty b_k < \infty$. Then, $\lim_{k \rightarrow \infty} X_k = 0$ a.s. and $\lim_{k \rightarrow \infty} \EXP{X_k} = 0$.
\end{lemma}
\begin{proof}
The given relation can be written  as $\EXP{X_{k+1} \mid X_0,\ldots,X_k} \leq X_k - \Tilde{\gamma}_k X_k + b_k$ a.s.\ for all $k\ge0$. We apply Robbins-Siegmund's super-martingale convergence theorem~\cite{robbins1971convergence} to the preceding relation to conclude that (i) $\lim_{k \rightarrow \infty} X_k = \bar X$ a.s., where $\bar X$ is some non-negative random variable, and (ii) $\sum_{k=0}^\infty \Tilde{\gamma_k} X_k < \infty$ a.s. Now from this last relation along with the fact that $\sum_{k=0}^\infty \Tilde{\gamma_k} = \infty$, we can conclude that $\bar X = 0$ a.s., and hence $\lim_{k \rightarrow \infty} X_k = 0$ a.s. Taking the total expectation on both sides of the relation 
$\EXP{X_{k+1} \mid X_0,\ldots,X_k} \leq X_k - \Tilde{\gamma}_k X_k + b_k$ 
and following similar arguments
as above, we obtain $\lim_{k \rightarrow \infty} \EXP{X_k} = 0$.
\end{proof}
\begin{lemma}\label{lem_near_super_martingale2}
    Let $\{Y_k\}$ and $\{Z_k\}$ be two sequences of non-negative random variables such that $Y_k \leq Z_k$ for all $k \geq 0$, $\lim_{k \rightarrow \infty} Z_k = 0$ a.s.\ and $\lim_{k \rightarrow \infty} \EXP{Z_k} = 0$. Then, the sequence $\{Y_k\}$ also converges to $0$ a.s.\ and in expectation.
\end{lemma}
\begin{proof}
    The random variables $Y_k$ are non-negative
    implying that $\EXP{Y_k}\ge0$, and
    \begin{align}
      0\le \liminf_{k\to\infty} Y_k\quad{\it a.s.},\qquad 0\le \liminf_{k\to\infty} \EXP{Y_k}. \label{eq-ykzk}
    \end{align}
    Since $ Y_k \leq Z_k$ for all $k \geq 0$, by taking the limit superior in this relation and using the fact that $\lim_{k \rightarrow \infty} Z_k = 0$ almost surely, we obtain $\limsup_{k \rightarrow \infty} Y_k \leq \lim_{k \rightarrow \infty} Z_k = 0$ almost surely. From this condition and \eqref{eq-ykzk}, we conclude that $\lim_{k \rightarrow \infty} Y_k = 0$ almost surely. Additionally, since $Y_k \leq Z_k$ for all $k \geq 0$, it follows that $\EXP{Y_k} \leq \EXP{Z_k}$ for all $k \geq 0$. So using $\lim_{k \rightarrow \infty} \EXP{Z_k} = 0$, we see that $\limsup_{k \rightarrow \infty} \EXP{Y_k} \leq \lim_{k \rightarrow \infty} \EXP{Z_k} = 0$. This relation along with~\eqref{eq-ykzk} imply that $\lim_{k \rightarrow \infty} \EXP{Y_k}=0$.
\end{proof}



\section{Randomized Feasibility Updates}\label{sec:inf_updates}
Here, we consider the randomized feasibility updates following~\cite{polyak2001random,nedic2019random,necoara2021minibatch}. The basic idea of the scheme is to randomly sample some sets $X_{a_j}$ out of the family $\{X_{a_j}, a_j\in\A_j\}$ for all $j \in \cJ$ and, then, sequentially perform feasibility updates. 
Unlike~\cite{wang2015incremental, cui2021analysis} that use the direct projection on a randomly selected set, we consider randomized feasibility updates since the projection on the functional constraints need not have a closed form expression.
 
 We next present a result for one step feasibility update for a functional constraint.
\begin{lemma}\label{lemma:basiter}
   Let $h$ be a convex function over a nonempty convex closed set $Z$. Given a vector $z\in Z$, a non-zero direction $d\in\partial h^+(z)$, and step size $\beta>0$, let $\hat z$ be given by
   \begin{align}
       \hat z=\Pi_{Z}\left[z-\b \frac{h^+(z)}{\|d\|^2}\,d\right].\label{inf_upd}
   \end{align}
    Then, for any $\bar z\in Z$ such that $h^+(\bar z)=0$, we have
     \[\|\hat z -\bar z\|^2 \le  \|z -\bar z\|^2
       -\b(2-\beta)\,\frac{(h^+(z))^2}{\|d\|^2}.\]
\end{lemma}
The proof of Lemma~\ref{lemma:basiter} can be found in~\cite[Theorem~1]{polyak1969minimization}; also see~\cite{polyak2001random}. For $0<\beta<2$, Lemma~\ref{lemma:basiter}
implies that the point $\hat z$
obtained via~\eqref{inf_upd} is closer to the level set $\{\widetilde z \in Z \mid h(\widetilde z) \leq 0\}$ than the point $z$.

%
%
%
\begin{algorithm}
		\caption{Random Feasibility Steps (For Agent $j$)}
		\label{algo_proj_steps}
		\begin{algorithmic}[1]
			\REQUIRE{$v_{k,j}$, $N_{k,j}$, deterministic step size $0 < \beta < 2$}
                \STATE \textbf{Set:} $z_{k,j}^0=v_{k,j}$
                \FOR{$i=1,\ldots,N_{k,j}$}
                \STATE \textbf{Sample:} Choose a random index $\o_{k,j}^{i}$
             and compute a subgradient
                $d_{k,j}^{i}$ of  $g^+_{\o_{k,j}^{i}}(z)$ at the point $z=z_{k,j}^{i-1}$ 
                \STATE \textbf{Update:} $z_{k,j}^i
   = \Pi_{Y_j} \left[z_{k,j}^{i-1} - \beta\, \frac{g^+_{\o_{k,j}^{i}}(z_{k,j}^{i-1})}{\|d_{k,j}^{i}\|^2}\, d_{k,j}^{i}\right]$
                \ENDFOR
                \STATE \textbf{Output} $x_{k,j}=z_{k,j}^{N_{k,j}}$ .
		\end{algorithmic}
\end{algorithm}	
Algorithm \ref{algo_proj_steps} performs feasibility updates for the functional constraints defining the set $X_j$. We will combine it with  the projection method (Algorithm~\ref{algo_proj_method}), Korpelevich method (Algorithm~\ref{algo_Kor_method}), and Popov method (Algorithm~\ref{algo_Popov_method}), which provide updates using the mapping $F$. The inputs to 
Algorithm \ref{algo_proj_steps} are the iterate $v_{k,j}$ (coming from Algorithms \ref{algo_proj_method}, \ref{algo_Kor_method}, or \ref{algo_Popov_method}) and a (deterministic) time varying batch size $N_{k,j} \geq 1$ at time $k$ for agent $j$. The iterate $x_{k,j}$ is the output  of Algorithm~\ref{algo_proj_steps}. 

With a deterministic step size $\beta\in(0,2)$, the Algorithm~\ref{algo_proj_steps} takes $N_{k,j}$ feasibility steps, each of which is random i.e.,
 the index variable $\o_{k,j}^{i} \in \A_j$ is random for all $i=1,\ldots, N_{k,j}$.
 The algorithm uses
 $d_{k,j}^{i}\in\partial g^+_{\o_{k,j}^{i}}(z_{k,j}^{i-1})$ if $g^+_{\o_{k,j}^{i}}(z_{k,j}^{i-1})>0$ and, otherwise, it can use $d_{k,j}^{i}=d$
 for some fixed vector $d\ne0$ (if $g^+_{\o_{k,j}^{i}}(z_{k,j}^{i-1})=0)$. We note that the choice of the vector $d\ne0$ is nonessential, since
 $z_{k,j}^{i}=z_{k,j}^{i-1}$ for any $d\ne0$ due to
 $g^+_{\o_{k,j}^{i}}(z_{k,j}^{i-1})=0$. 
 By Assumption \ref{asum_closed_set}, $\partial g_{\omega_{k,j}^i}^+(x_j)\ne\emptyset$ for all $x_j\in \re^{n_j}$, $\omega_{k,j}^i\in\A_j$, $k \geq 1$, $i = 1, \ldots, N_{k,j}$, and $j = 1, \ldots, J$. Therefore, the feasibility updates are well defined. We let $x_{0,j}\in Y_j$ be a randomly selected starting point (to be used in Algorithms~\ref{algo_proj_method}, \ref{algo_Kor_method}, and \ref{algo_Popov_method}) for agent $j$ such that $\EXP{\|x_{0,j}\|^2}<\infty$. 
Our next assumption deals with the random variables $\o_{k,j}^i$ similar to \cite{nedic2019random,necoara2022stochastic}.
%
%
\begin{assumption}\label{asum-regularmod} 
There exists a constant $c>0$ such that for all $j\in \cJ$, 
 \[\dist^2(x_j,S_j)\le c\,\EXP{(g^+_{\o_{k,j}^i}(x_j))^2}  \quad\hbox{for all } x_j \in Y_j,\quad i=1,\ldots, N_{k,j}. \]
 \end{assumption}
Regarding the sampling of the random indices $\o_{k,j}^i$, Assumption~\ref{asum-regularmod} allows for independent and identically distributed sampling, as well as sampling with distribution conditional on the past samples drawn at iteration $k$. 
The constant $c$ in Assumption~\ref{asum-regularmod} can be seen to exists, for example, if the index set $\A_j$ is bounded and
there is a global Hoffman's error bound for the function
\[f_j(x_j)=\sup_{a_{j}\in\A_j} g_{a_j}(x_j) +\delta_{Y_j}(x_j),\]
where $\delta_C$ is the characteristic function of the set $C$.
A function $f(z)$ has a a global Hoffman's error bound when there exists a constant $\gamma>0$ such that
\[\dist(z,[f\le0])\le \gamma f^+(z)\qquad\hbox{for all }z,\]
where $[f\le0]$ is the lower-level set of $f$ associated with the zero value. When the index set $\A_j$ is finite, the set $Y_j$ is polyhedral, and each $g_{\a_j}$ is a linear function, then such an error bound exists
\cite[Lemma 3]{hoffman2003approximate}, \cite[Theorem 3.1]{li2013global}. Also, such a bound exists for quadratic and more general convex functions under a Slater condition \cite{luo1994extension,lewis1998error,li2013global}, as well as for infinitely many linear constraints~\cite{hu1989approximate}. Whenever a global Hoffman's error bound exists, Assumption~\ref{asum-regularmod}
holds with a constant $c$ that depends on the error bound $\gamma$ and the distribution used for random constraint sampling, assuming that the support of the distribution coincides with the index set.

In the sequel, we use the sigma algebra $\mathcal{F}_{k,j}$ per agent $j$ and their union $\widetilde \cF_{k}$ over agents, defined as follows:
  \begin{align} 
      \mathcal{F}_{k,j} = \{x_{0,j}\} \cup \{\omega_{t,j}^{i} \mid 1 \leq t \leq k, 1 \leq i \leq N_{k,j}\} , \qquad \widetilde \cF_{k} = \bigcup_{j=1}^J \mathcal{F}_{k,j} . \nonumber
  \end{align}
Also, we will use the quantity $q\in(0,1]$, defined as follows:
    \begin{align}\label{quantity-q}
    q=\frac{\b(2-\beta)}{c M_g^2 }.
    \end{align}
We next present a basic relation for the iterates of Algorithm~\ref{algo_proj_steps}.

\begin{theorem}\label{thm_inf_updates}
    Under Assumptions \ref{asum_closed_set}, \ref{asum-subgrad-bound}, and \ref{asum-regularmod}, for 
   the iterates $x_{k,j}$ obtained via Algorithm \ref{algo_proj_steps} we have 
    for all $x_j \in S_j = X_j \cap Y_j$ and  $j \in \{1,\ldots, J\}$,
    \[\|x_{k,j} - x_j\|^2 \le \|v_{k,j} - x_j\|^2
       - \frac{\b(2-\beta)}{M_g^2} \,\sum_{i=1}^{N_{k,j}} (g_{\o_{k,j}^{i}}^+(z_{k,j}^{i-1}))^2,\]
       and almost surely
    \begin{align}
        \EXP{\|x_{k,j}-x_j\|^2 \!\mid \! \cF_{k-1,j}} \! \leq \! \|v_{k,j}-x_j\|^2 \!-\! \left( (1-q)^{-N_{k,j}} \!-\! 1  \right) \EXP{\dist^2(x_{k,j},S_j) \mid  \cF_{k-1,j}} , \nonumber
    \end{align}
    where $0< \b < 2$, $M_g$ is the bound on subgradient norms from Assumption~\ref{asum-subgrad-bound}, and $q$ as given in~\eqref{quantity-q}.
\end{theorem}
\begin{proof}
We use the definition of $z_{k,j}^i$ in Algorithm \ref{algo_proj_steps} and Lemma~\ref{lemma:basiter}, with $Z=Y_j$. Thus, for all $x_j \in S_j = X_j \cap Y_j$, we obtain
for all  $i=1,\ldots,N_{k,j}$,
    \begin{align}\label{eq_infes}
        \|z_{k,j}^i - x_j\|^2 
        &\le  \|z_{k,j}^{i-1} - x_j\|^2
       -\b(2-\beta)\,\frac{(g_{\o_{k,j}^{i}}^+(z_{k,j}^{i-1}))^2}{\|d_{k,j}^{i}\|^2}\cr
       &\le \|z_{k,j}^{i-1} - x_j\|^2
       -\frac{\b(2-\beta)}{M_g^2}\,(g_{\o_{k,j}^{i}}^+(z_{k,j}^{i-1}))^2,
    \end{align}
    where the last inequality is implied by $\|d_{k,j}^i\|^2 \leq M_g^2$
    for all $k$ and $j$ (Assumption~\ref{asum-subgrad-bound}).
       By summing the preceding relations over $i=1,\ldots,N_{k,j}$, 
       and by using $z^0_{k,j}=v_{k,j}$ and $z^{N_{k,j}}_{k,j}=x_{k,j}$, 
\begin{align}
        \|x_{k,j} - x_j\|^2 \le \|v_{k,j} - x_j\|^2
       - \frac{\b(2-\beta)}{M_g^2} \,\sum_{i=1}^{N_{k,j}} (g_{\o_{k,j}^{i}}^+(z_{k,j}^{i-1}))^2, \label{x_bound}
    \end{align}
    which implies the first stated relation upon dropping the last nonpositive term.
Next, taking minimum with respect to $x_j \in S_j$ on both the sides of~\eqref{eq_infes} gives us 
\begin{align*}
    \dist^2(z_{k,j}^i,S_j) \leq \dist^2(z_{k,j}^{i-1},S_j) - \frac{\b(2-\beta)}{M_g^2} (g_{\o_{k,j}^{i}}^+(z_{k,j}^{i-1}))^2 \quad \forall i=1,\ldots,N_{k,j}, \; j \in \cJ.
\end{align*}
Taking the conditional expectation in the preceding relation, given the past sigma algebra $\cF_{k-1,j}$, we obtain almost surely for all $i=1,\ldots,N_{k,j}$,
\begin{align}
    \EXP{\dist^2(z_{k,j}^i,S_j) \mid \cF_{k-1,j}} \leq \;&\EXP{\dist^2(z_{k,j}^{i-1},S_j) \mid \cF_{k-1,j}} \nonumber\\
    & - \frac{\b(2 - \beta )}{M_g^2} \EXP{(g_{\o_{k,j}^{i}}^+ (z_{k,j}^{i-1}))^2 \mid \cF_{k-1,j}} . \label{eq_infes2}
\end{align}
By using Assumption~\ref{asum-regularmod} for the last term in the preceding relation, we obtain almost surely for all $i=1,\ldots,N_{k,j}$, $j\in\cJ$,
     \begin{align}
         \EXP{(g_{\o_{k,j}^{i}}^+(z_{k,j}^{i-1}))^2 \mid \cF_{k-1,j}} &= \EXP{\EXP{(g_{\o_{k,j}^{i}}^+(z_{k,j}^{i-1}))^2 \mid \cF_{k-1,j},z_{k,j}^{i-1}}} \nonumber\\
         &\geq  \frac{1}{c} \EXP{\dist^2(z_{k,j}^{i-1},S_j) \mid \cF_{k-1,j}}. \label{eq-startexpest}
     \end{align}
Combining~\eqref{eq-startexpest} with \eqref{eq_infes2}, we can see that almost surely for all $i=1, \ldots, N_{k,j}$,
\begin{align}
    \EXP{\dist^2(z_{k,j}^i,S_j) \mid \cF_{k-1,j}} \leq \left(1 - \frac{\b(2-\beta)}{c M_g^2}\right) \EXP{\dist^2(z_{k,j}^{i-1},S_j) \mid \cF_{k-1,j}}, 
\end{align} 
where
$q$ is  given by~\eqref{quantity-q},
which satisfies $q\ge0$ since $\b\in(0,2)$. Note that the preceding relation implies that we must have $1-q\ge0$. If $q=1$, it would follow that 
$ \EXP{\dist^2(z_{k,j}^i,S_j) \mid \cF_{k-1,j}} =0$ for all $i=1,\ldots,N_{k,j}$,
which would imply that $\dist^2(z_{k,j}^i,S_j)=0$ almost surely for all $i=1,\ldots,N_{k,j}$, which is highly unlikely case. Thus, without loss of generality, we assume that  $q<1$.
Using the definition of $x_{k,j}$, i.e., $x_{k,j}=z_{k,j}^{N_{k,j}}$,
      we have almost surely for all $i=1,\ldots,N_{k,j}$,
      \[\EXP{\dist^2(x_{k,j},S_j)\mid \cF_{k-1,j}}\le
      (1-q)^{N_{k,j}-i+1} \EXP{\dist^2(z_{k,j}^{i-1},S_j) \mid \cF_{k-1,j}} \qquad  {a.s.}\]
     Using~\eqref{eq-startexpest} in the preceding relation, we obtain
     for all $i=1,\ldots,N_{k,j}$,
     \[ \EXP{g_{\o_{k,j}^{i}}^+(z_{k,j}^{i-1}))^2 \mid \cF_{k-1,j}}
     \ge \frac{1}{c}\frac{1}{(1-q)^{N_{k,j}-i+1}}\EXP{\dist^2(x_{k,j},S_j) \mid \cF_{k-1,j}} \qquad  {a.s}.\]
     Summing over $i=1,\ldots,N_{k,j}$, and simplifying $\sum\limits_{i=1}^{N_{k,j}} \frac{1}{(1-q)^{N_{k,j}-i+1}} = \frac{1- (1-q)^{N_{k,j}}}{q(1-q)^{N_{k,j}}}$, we get
     \begin{align}
         \sum_{i=1}^{N_{k,j}} \EXP{g_{\o_{k,j}^{i}}^+(z_{k,j}^{i-1}))^2 \mid \cF_{k-1,j}}
     \ge \frac{1}{c}\frac{\left(1- (1-q)^{N_{k,j}}\right)}{q(1-q)^{N_{k,j}}}  \EXP{\dist^2(x_{k,j},S_j) \mid \cF_{k-1,j}} \;\; a.s. \nonumber
     \end{align}
     Since $q=\frac{\b(2-\beta)}{c M_g^2 }$,  we multiply by $\frac{\b(2-\beta)}{M_g^2 }$ both sides of the above inequality and simplify the relation to obtain almost surely
     \begin{align}
     \frac{\b(2-\beta)}{M_g^2} \! \sum_{i=1}^{N_{k,j}} \EXP{(g_{\o_{k,j}^{i}}^+(z_{k,j}^{i-1}))^2 \!\mid \! \cF_{k-1,j}}
     & \!\ge\! \left( (1\!-\!q)^{-{N_{k,j}}} - 1\right)  \EXP{\dist^2(x_{k,j},S_j) \!\mid \! \cF_{k-1,j}}. \nonumber
     \end{align}
 Now taking conditional expectation on both sides of \eqref{x_bound} with respect to $\cF_{k-1,j}$, and then using the preceding relation for the last term on the right hand side of~\eqref{x_bound}, we obtain the relation of Theorem \ref{thm_inf_updates}.
\end{proof}
%
%
%
%

Next, we present a lemma that shows a bound of the expected distance of the iterate $x_{k,j}$ from its projection on the set $S_j$ corresponding to each agent $j \in \cJ$.
\begin{lemma}\label{lem_inf_geom}
Under Assumptions \ref{asum_closed_set}, \ref{asum-subgrad-bound}, and \ref{asum-regularmod}, for 
   the iterate $x_{k,j}$ obtained by Algorithm \ref{algo_proj_steps} we have almost surely
    for all agents  $j \in \{1,\ldots, J\}$,
    \begin{align*}
        &\EXP{\dist^2(x_{k,j},S_j) \mid \cF_{k-1,j}} \leq (1-q)^{N_{k,j}} {\|v_{k,j} - x_j^*\|^2},
    \end{align*}
    where $x^*=(x_1^*,\ldots, x_J^*)$ is the solution of VI$(S, F)$ and $q$ is given in \eqref{quantity-q}.
\end{lemma}
\begin{proof}
By letting $x_j = x_j^* \in X_j \cap Y_j$ in Theorem \ref{thm_inf_updates} for the agent $j$, where $x^*$ is the solution to the VI problem in~\eqref{VI_problem}, we obtain almost surely 
\begin{align*}
\EXP{\|x_{k,j} - x_j^*\|^2 \mid \cF_{k-1,j}}  
&\le  \|v_{k,j} - x_j^*\|^2 \cr
&\ \ - \left( (1-q)^{-{N_{k,j}}}-1\right)\EXP{\dist^2(x_{k,j},S_j) \mid \cF_{k-1,j}}. \label{eq_dist1}
\end{align*}
Using the relation $\EXP{\dist^2(x_{k,j},S_j) \mid \cF_{k-1,j}} \leq \EXP{\|x_{k,j} - x_j^*\|^2 \mid \cF_{k-1,j}}$ a.s.\ in the preceding inequality and grouping together the terms of $\EXP{\dist^2(x_{k,j},S_j) \mid \cF_{k-1,j}}$ on the left hand side, we obtain
    \begin{align}
        (1-q)^{-N_{k,j}} \EXP{\dist^2(x_{k,j},S_j) \mid \cF_{k-1,j}} \leq \|v_{k,j} - x_j^*\|^2 \qquad  {a.s}.,
    \end{align}
 \end{proof}

    We note that by using Lemma~\ref{lem_inf_geom}
and Jensen's inequality
    \begin{align*}
        \EXP{\dist(x_{k,j},S_j) \!\mid \! \cF_{k-1,j}} 
        & \!=\! \EXP{\sqrt{\dist^2(x_{k,j},S_j)} \!\mid \! \cF_{k-1,j}}\cr 
        &\!\leq \! \sqrt{\EXP{\dist^2(x_{k,j},S_j) \!\mid \! \cF_{k-1,j}}},
    \end{align*}
    we can also obtain that almost surely
    for all $j\in\mathcal{J},$
    \[\EXP{\dist(x_{k,j},S_j) \mid \cF_{k-1,j}} \leq (1-q)^{\frac{N_{k,j}}{2}} \|v_{k,j} - x_j^*\|.\]

%
%
%
To consider the relation for all the agents together, we concatenate the vectors 
$v_{k,j}, x_{k,j}$ across all the agents to define
$v_k$ and $x_k$, defined as follows
\begin{align}
    v_k= [v_{k,1}, \ldots, v_{k,J}]^T , \qquad
    x_k= [x_{k,1}, \ldots, x_{k,J}]^T, \label{stacked_vector}
\end{align}
and sum the relations in Theorem~\ref{thm_inf_updates} and  Lemma~\ref{lem_inf_geom} for all $j=1,\cdots,J$ to obtain
\begin{align}
    \EXP{\|x_{k} - x^*\|^2 \mid \widetilde \cF_{k-1}}  \le&  {\|v_{k} - x^*\|^2} \nonumber\\
       &- \sum_{j=1}^J \left( (1-q)^{-{N_{k,j}}}-1\right)\EXP{\dist^2(x_{k,j},S_j) \mid \cF_{k-1,j}} \quad a.s. ,  \label{dist_eq_opt} \\
    &\EXP{\dist^2(x_{k},S) \mid \widetilde \cF_{k-1}} \leq \sum_{j=1}^J (1-q)^{N_{k,j}} {\|v_{k,j} - x_j^*\|^2} \quad a.s. \label{dist_eq_geom}
\end{align}

We next use Algorithm~\ref{algo_proj_steps} to develop several VI algorithms to solve~\eqref{VI_problem}.


\section{Modified Projection Method}\label{sec:proj_method}
The modified projection method (Algorithm \ref{algo_proj_method}), at iteration $k$, updates the point $x_{k-1,j}$ using the mapping $[F(x_{k-1})]_j$ to produce the iterate $v_{k,j}$ for all agents $j \in \cJ$. 
In a game setting,
the mapping block $[F(x_{k-1})]_j$ for agent $j$ corresponds to $\nabla_{x_j} \theta_j(x_{k-1})$ of the agent's cost function. After performing the update step, Algorithm \ref{algo_proj_steps} is called to reduce the infeasibility gap between $v_{k,j}$ and the set $X_j$, giving $x_{k,j}$ as the output at iteration $k$ for all $j \in \cJ$. The same step size $\a_k$ is used across all the agents, which satisfies the assumption below.
%
%
%
%
\begin{algorithm}
		\caption{Modified Projection Method (For Agent $j$)}
		\label{algo_proj_method}
		\begin{algorithmic}[1]
			\REQUIRE{$x_{0,j}$, $\a_k$}
                \FOR{$k=1,\ldots$}
                \STATE \textbf{Update:} $v_{k,j} = \Pi_{Y_j} \left[x_{k-1,j} - \a_{k-1} [F(x_{k-1})]_j\right]$
                \STATE \textbf{Call Algorithm \ref{algo_proj_steps}:} Pass $v_{k,j}$ and $N_{k,j}$ and get $x_{k,j}$
                \ENDFOR
		\end{algorithmic}
\end{algorithm}	
Using the stacked vector notation (see \eqref{stacked_vector}),
we provide a relation for the distance $\|v_{k+1} - x^*\|^2$ at iteration $k+1$ from the solution $x^*$ of the VI$(S, F)$. 
\begin{lemma}\label{lem_proj}
    Under Assumptions~\ref{asum_closed_set},~\ref{asum_lipschitz}, and~\ref{asum_strong_monotone}, 
    and assuming that the step size $\alpha_k$ satisfies $0<\a_k\le \mu/(2L^2)$,
    the following holds  almost surely for the iterates $v_{k+1},x_k \in Y$ generated by Algorithm~\ref{algo_proj_method} and the solution $x^*$,
    \begin{align*}
&\EXP{\| v_{k+1} - x^* \|^2 \mid \widetilde \cF_{k-1}} 
\leq  ( 1 - \a_k \mu) \| v_{k} - x^* \|^2 + \a_k^2 \left(2 + \frac{1}{\zeta_k} \right) \| F(x^*) \|^2 \nonumber\\
        & \ \ - \sum_{j=1}^J \left[( 1 - \a_k \mu)((1-q)^{-N_{k,j}}-1) - \zeta_k \right] \EXP{\dist^2(x_{k,j},S_j) \mid \cF_{k-1,j}}, 
\end{align*}
    where $\zeta_k > 0$ is a scalar and the constant $q$ is given by \eqref{quantity-q}.
\end{lemma}
\begin{proof}
    From the definition of $v_{k+1} \in Y$ and using the non-expansiveness property of the projection operator, we can see
    \begin{align}
        \| v_{k+1} - x^* \|^2 &=  \| \Pi_{Y} [x_{k} - \alpha_k F(x_k)] - x^* \|^2 \leq \| (x_{k} - x^*) - \alpha_k F(x_k) \|^2 \nonumber \\
        &= \| x_{k} - x^* \|^2 - 2 \alpha_k \la F(x_k), x_k - x^* \ra + \alpha_k^2 \|F(x_k)\|^2 . \label{lem2_bound1}
    \end{align}
We use Assumptions~\ref{asum_strong_monotone} (strong monotonicity) and~\ref{asum_lipschitz} (Lipschitz continuity) to upper bound the second and third terms respectively on the right hand side of \eqref{lem2_bound1}:
\begin{align}
    & - 2 \alpha_k \la F(x_k), x_k - x^* \ra \leq - 2 \alpha_k \la F(x^*), x_k - x^* \ra - 2 \alpha_k \mu \| x_k - x^* \|^2 .\label{lem2_bound2} \\
    &\| F(x_k) \|^2 \leq 2\|F(x_k) - F(x^*) \|^2 + 2\|F(x^*) \|^2 \leq 2L^2 \| x_k - x^* \|^2 + 2 \|F(x^*) \|^2 . \label{lem2_bound5}
\end{align}
The first quantity on the right hand side of \eqref{lem2_bound2} can be upper estimated as
\begin{align}
    - 2 \alpha_k \la F(x^*), x_k - x^* \ra = 2 \alpha_k \la F(x^*), x^* - \Pi_S[x_k] \ra + 2 \alpha_k \la F(x^*), \Pi_S[x_k] - x_k \ra . \label{lem5_bound2}
\end{align}
The first term on the right hand side of \eqref{lem5_bound2} is non-positive since $x^*$ is the solution of the VI$(S,F)$ and the second term can be estimated using Young's inequality:
\begin{align}
    - 2 \alpha_k \la F(x^*), x_k - x^* \ra \leq \frac{\a_k^2}{\zeta_k} \| F(x^*) \|^2 + \zeta_k \dist^2(x_k,S) , \label{lem5_bound3}
\end{align}
where $\zeta_k>0$ is an arbitrary scalar. Using relations~\eqref{lem5_bound3} and ~\eqref{lem2_bound2}, and substituting the resulting relation back into~\eqref{lem2_bound1}, we obtain 
\begin{align}
        \| v_{k+1} - x^* \|^2 \!\leq\! ( 1 \!-\! 2\a_k \mu \!+\! 2 \a_k^2 L^2) \| x_k - x^* \|^2 \!+\! \left(2 \!+\! \frac{1}{\zeta_k} \right) \a_k^2 \| F(x^*) \|^2 \!+\! \zeta_k \dist^2(x_k,S).  \nonumber
    \end{align}
 For the step size $\a_k>0$ satisfying $\a_k\leq \mu/(2 L^2)$, we have that 
    \[1 - 2\a_k \mu + 2 \a_k^2 L^2 \leq 1 - \a_k \mu,\]
implying that
\begin{align}
    \| v_{k+1} - x^* \|^2 \leq ( 1 - \a_k \mu) \| x_k - x^* \|^2 + \left(2 \!+\! \frac{1}{\zeta_k} \right) \a_k^2 \| F(x^*) \|^2 + \zeta_k \dist^2(x_k,S) . \label{Proj_meth_eq2}
\end{align}
Taking the conditional expectation in the preceding relation, given $\widetilde \cF_{k-1}$, and using~\eqref{dist_eq_opt} to simplify the first term on the right hand side of the above expression gives us
\begin{align*}
&\EXP{\| v_{k+1} - x^* \|^2 \mid \widetilde \cF_{k-1}} 
\leq  ( 1 - \a_k \mu) \| v_{k} - x^* \|^2 + \a_k^2 \left(2 + \frac{1}{\zeta_k} \right) \| F(x^*) \|^2 \nonumber\\
        & \ \ - \sum_{j=1}^J \left[( 1 - \a_k \mu)((1-q)^{-N_{k,j}}-1) - \zeta_k \right] \EXP{\dist^2(x_{k,j},S_j) \mid \cF_{k-1,j}}  \quad a.s., 
\end{align*}
which is the stated relation. 
\end{proof}
We note that $\zeta_k$ in Lemma~\ref{lem_proj} is an artifact of the analysis and not a parameter of Algorithm~\ref{algo_proj_steps}. As seen from
Lemma~\ref{lem_proj}, we want to select $\zeta_k>0$ so that $( 1 - \a_k \mu)((1-q)^{-N_{k,j}}-1) \geq \zeta_k$, which gives
\begin{align}
    N_{k,j} \geq \frac{\log\left(\frac{ \zeta_k+1-\a_k\mu}{1-\a_k\mu}\right)}{\log\left(\frac{1}{1-q}\right)} = \log_{\frac{1}{1-q}}\left(\frac{ \zeta_k+1-\a_k\mu}{1-\a_k\mu}\right) \;, \quad \forall k\geq 1 \;, j \in \cJ. \label{sample_selection}
\end{align}
We want our algorithm to work for any batch size, i.e. $N_{k,j} \geq 1$, and hence we choose $\xi_k$ such that 
$\frac{ \zeta_k+1-\a_k\mu}{1-\a_k\mu} = \frac{1}{1-q}$, which gives $\zeta_k$ as
\begin{align}
    \zeta_k = \frac{q}{1-q}(1-\a_k \mu) , \label{zeta_selection}
\end{align}
which is positive for all $k \geq 1$ when $1-\a_k \mu > 0$, thus restricting the stepsize $\a_k < \frac{1}{\mu}$. 
This condition is satisfied when $\a_k \leq \frac{\mu}{2L^2}$ since $\frac{\mu}{2L^2} < \frac{1}{\mu}$. Hence, under the stepsize choice
as in Lemma~\ref{lem_proj} and using $\zeta_k$
as in~\eqref{zeta_selection}, the batch size satisfies $N_{k,j} \geq 1$ for all $k \geq 1$.


We next show almost sure convergence of the iterates $v_k \in Y$ and $x_k \in Y$ to the solution $x^* \in S$ and also almost sure convergence of the infeasibility gap $\dist(x_k,S)$ to $0$.
These results hold when the stepsize satisfies the following assumption.
\begin{assumption}\label{asum_step_size}
    The step size $\a_k>0$ satisfies $\sum_{k=0}^\infty \a_k = \infty$ and $\sum_{k=0}^\infty \a_k^2 < \infty$.
\end{assumption}


\begin{theorem}\label{Proj_meth_thm_as_conv}
    Under Assumptions \ref{asum_closed_set},~\ref{asum-subgrad-bound},~\ref{asum_lipschitz},~\ref{asum_strong_monotone},~\ref{asum-regularmod}, and \ref{asum_step_size}, with a step-size $\a_k \leq \frac{\mu}{2 L^2}$ for all $k$,  $\zeta_k$ as in~\eqref{zeta_selection}, and any choice of the batch sizes $N_{k,j}\ge 1$ for all $k$ and $j$, the following results hold for Algorithm \ref{algo_proj_method}:\\
    (i) the iterates $v_{k}, x_{k} \in Y$ converge to $x^*$ almost surely and in expectation,\\
    (ii) the distance $\dist(x_k,S)$ converges to $0$ almost surely and in expectation. 
\end{theorem}
\begin{proof}
We use Lemma~\ref{lem_proj}
    and select $\zeta_k$ as \eqref{zeta_selection} which ensures the last term in the relation of Lemma~\ref{lem_proj} stays nonpositive for any $N_{k,j} \geq 1$ for all $k,j$. Therefore, 
    \begin{align}
         \EXP{\| v_{k+1} - x^* \|^2 \mid \widetilde \cF_{k-1}} &\leq (1-\a_k \mu) \| v_{k} - x^* \|^2 + \a_k^2 \left(2 + \frac{1}{\zeta_k} \right) \| F(x^*) \|^2 .\label{Proj_meth_thm_eq3} 
    \end{align}
    Now, with the step size $\a_k \leq \frac{\mu}{2L^2}$, we obtain $1-\a_k \mu \geq 1-\frac{\mu^2}{2L^2}$ for all $k$. Hence, using the preceding relation and equation \eqref{zeta_selection}, we see
    \begin{align}
        \frac{1}{\zeta_k} \leq \frac{1-q}{q(1-\frac{\mu^2}{2L^2})} . \label{zeta_selection2}
    \end{align}
    Therefore, using Assumption~\ref{asum_step_size} and the relation \eqref{zeta_selection2}, we obtain $\sum_{k=0}^\infty \a_k^2/\zeta_k < \infty$. By using Lemma~\ref{lem_near_super_martingale}, with the identification $\Tilde{\gamma}_k = \a_k \mu$, $X_k = \|v_k-x^*\|^2$, and $b_k = \left(2 + \frac{1}{\zeta_k} \right) \a_k^2 \| F(x^*) \|^2$, we conclude that 
    \begin{align}
        &\lim_{k \rightarrow \infty} \|v_k - x^*\|^2 = 0 \quad a.s., \quad  \quad \text{and} \quad \lim_{k \rightarrow \infty} \EXP{\|v_k - x^*\|^2} = 0.\label{mart_proj1}
    \end{align}
    Using Jensen's inequality, we have $\EXP{\|v_k - x^*\|} = \EXP{\sqrt{\|v_k - x^*\|^2}} \leq \sqrt{\EXP{\|v_k - x^*\|^2}}$. So using the preceding relation and the second expression of \eqref{mart_proj1}, we obtain $\lim_{k \rightarrow \infty} \EXP{\|v_k - x^*\|} = 0$.
    From Theorem~\ref{thm_inf_updates}, we can see that $\|x_k - x^* \| \leq \|v_k - x^* \|$ by summing both the sides of the first stated inequality for $j=1, \ldots, J$. Using this inequality and Lemma~\ref{lem_near_super_martingale2}, we can see that $\lim_{k \rightarrow \infty} \|x_k - x^* \| = 0$ a.s.\ and $\lim_{k \rightarrow \infty} \EXP{\|x_k - x^* \|} = 0$.
    Since $\dist(x_k,S) = \|x_k - \Pi_S[x_k]\| \leq \|x_k - x^*\|$ for all $k$. Hence, using Lemma~\ref{lem_near_super_martingale2} and the preceding relations, we conclude that    $ \lim_{k \rightarrow \infty} \dist(x_k,S) = 0$ a.s.\ and $\lim_{k \rightarrow \infty} \EXP{\dist(x_k,S)} = 0$.
\end{proof}
We next present the convergence rate results for Algorithm~\ref{algo_proj_method}.
\begin{theorem}\label{Proj_meth_thm_rate}
    Under the assumptions of Theorem~\ref{Proj_meth_thm_as_conv}
    and the step size satisfying $\a_k \leq \min \left(\frac{2}{\mu(k+1)}, \frac{\mu}{2 L^2} \right)$ for all $k$, the expected squared norm distances of $v_{k+1}$ and $x_{k+1}$ from the solution $x^*$ are bounded above by the order
    \begin{align}
        O \left(\frac{\kappa^4}{T^2} + \frac{1}{\mu^2 T} \|F(x^*)\|^2 \right) ,\nonumber
    \end{align}
    where $\kappa$ is the condition number (see~\eqref{eq-cnum}). Additionally, the expected distance of $x_{T+1}$ to the set $S$ diminishes geometrically fast with the sample size
    \begin{align}
        \EXP{\dist(x_{T+1},S)} &\leq \max_{j \in \cJ} (1-q)^{\frac{N_{T+1,j}}{2}} \times O \left(\frac{\kappa^2}{T} + \frac{1}{\mu \sqrt{T}} \|F(x^*)\| \right) , \nonumber
    \end{align}
    where $q$ is given by~\eqref{quantity-q}.
\end{theorem}
\begin{proof}
    Since the conditions of Theorem~\ref{Proj_meth_thm_as_conv} are satisfied, the relation \eqref{Proj_meth_thm_eq3} is valid almost surely.
    Upon taking the total expectation on both sides of~\eqref{Proj_meth_thm_eq3}, we obtain
    \begin{align}
        \EXP{\| v_{k+1} - x^* \|^2} &\leq (1-\a_k \mu) \EXP{\| v_{k} - x^* \|^2} + \a_k^2 \left(2 + \frac{1}{\zeta_k} \right) \| F(x^*) \|^2 . \label{eq_proj_thm_rate1}
    \end{align}
    We let $\gamma_k=\a_k \mu$, and chose $\a_k$ as follows
    \begin{align}
        \a_k \leq \min \left(\frac{2}{\mu(k+1)}, \frac{\mu}{2 L^2} \right) \qquad \hbox {implying that}\qquad
        \gamma_k \leq \min \left(\frac{2}{k+1}, \frac{\mu^2}{2 L^2} \right).
    \end{align}
    Hence \eqref{eq_proj_thm_rate1} can be written in terms of $\gamma_k$:
    \begin{align}
        \EXP{\| v_{k+1} - x^* \|^2} &\leq (1-\gamma_k) \EXP{\| v_{k} - x^* \|^2} + \frac{\gamma_k^2}{\mu^2} \left(2 + \frac{1}{\zeta_k} \right) \| F(x^*) \|^2 . \nonumber
    \end{align}
    Next, we divide by $\gamma_k^2$ both sides of the preceding equation and note that, under the above condition on the step size, we have that $\frac{1-\gamma_k}{\gamma_k^2} \leq \frac{1}{\gamma_{k-1}^2}$. Hence, it follows that
    \begin{align}
        \gamma_k^{-2} \EXP{\| v_{k+1} - x^* \|^2} \leq \gamma_{k-1}^{-2} \EXP{\| v_{k} - x^* \|^2} + \frac{1}{\mu^2} \left(2 + \frac{1}{\zeta_k} \right) \| F(x^*) \|^2 . \nonumber
    \end{align}
    Next, we sum the preceding relation from $k=1$ to some iterate $T\ge 1$ and use the constant upper bound on $1/\zeta_k$ from \eqref{zeta_selection2} to obtain
    \begin{align}
        \gamma_T^{-2} \EXP{\| v_{T+1} - x^* \|^2} \leq \gamma_{0}^{-2} \EXP{\| v_{1} - x^* \|^2} + \frac{T}{\mu^2} \left(2 + \frac{1-q}{q(1-\frac{\mu^2}{2L^2})} \right) \| F(x^*) \|^2 . \label{eq_proj_thm_rate2}
    \end{align}
    Note that $\gamma_0 = \min \left(2, \frac{\mu^2}{2 L^2}  \right)$. Now $\frac{\mu}{L} = \frac{1}{\kappa} \leq 1$ where $\kappa$ is the condition number (cf. \eqref{eq-cnum}). So, we obtain $\gamma_0 = \frac{\mu^2}{2 L^2}$, which implies $\gamma_0^{-2} = 4 \kappa^4$. Additionally, note that $\gamma_T^{-2} = \frac{(T+1)^2}{4}$, when $T>>1$. Hence, from the preceding relation we can obtain
    \begin{align}
        &\EXP{\| v_{T+1} - x^* \|^2} \leq O \left(\frac{\kappa^4}{T^2} + \frac{1}{\mu^2 T} \|F(x^*)\|^2 \right) . \nonumber
    \end{align}
    Now, taking the total expectation on both sides of the relation \eqref{dist_eq_opt}, we can see that $\EXP{\| x_{T+1} - x^* \|^2} \leq \EXP{\| v_{T+1} - x^* \|^2}$ and, hence, using the preceding estimate we find
    \begin{align}
        &\EXP{\| x_{T+1} - x^* \|^2} \leq O \left(\frac{\kappa^4}{T^2} + \frac{1}{\mu^2 T} \|F(x^*)\|^2 \right) . \nonumber
    \end{align}
    Now, we consider \eqref{dist_eq_geom} with $k=T+1$ and take the total expectation on both sides of the equation to obtain
    \begin{align}
        \EXP{\dist^2(x_{T+1},S)} &\leq \EXP{\|v_{T+1}-x^*\|^2} \max_{j \in \cJ} (1-q)^{N_{T+1,j}} \nonumber\\
        &\leq \max_{j \in \cJ} (1-q)^{N_{T+1,j}} \times O \left(\frac{\kappa^4}{T^2} + \frac{1}{\mu^2 T} \|F(x^*)\|^2 \right) . \nonumber
    \end{align}
    Using Jensen's inequality, we can finally conclude 
    \begin{align}
        \EXP{\dist(x_{T+1},S)} &= \EXP{\sqrt{\dist^2(x_{T+1},S)}} \leq \sqrt{\EXP{\dist^2(x_{T+1},S)}} \nonumber\\
        &\leq \max_{j \in \cJ} (1-q)^{N_{T+1,j}/2} \times O \left(\frac{\kappa^2}{T} + \frac{1}{\mu \sqrt{T}} \|F(x^*)\| \right) . \nonumber
    \end{align}
\end{proof}
Theorem \ref{Proj_meth_thm_rate} shows that for a diminishing step size, we obtain $O(1/T)$ rate of convergence for $\EXP{\| x_{T+1} - x^* \|^2}$ and a geometric convergence with respect to the number of samples times $O(1/\sqrt{T})$ for $\EXP{\dist(x_{T+1},S)}$. If $N_{k,j} = 1$ for all $k,j$, then $\EXP{\dist(x_{T+1},S)} \leq O(1/\sqrt{T})$. If we fix the iteration number $T$, then we can run the previous iterations with low sample size and in the last iteration, we can take a large $N_{k,j}$ so that the infeasibility gap decreases exponentially fast. Alternatively,
we may grow the batch sizes $N_{k,j}$ in the order of $\log k$ at time $k$ for all $j$.


\section{Modified Korpelevich Method}\label{sec:Korpelevich_method}
Here, we consider a modification of the Korpelevich method \cite{korpelevich1976extragradient} that has improved stability~\cite{mokhtari2020unified} as compared to the standard projection method for VIs.
The modified Korpelevich method (Algorithm \ref{algo_Kor_method}) at iteration $k \geq 1$ uses the same step size $\a_k$ for all the agents $j \in \cJ$ and performs two updates. 
Firstly, it updates the iterate $x_{k-1,j}$ using the mapping $[F(x_{k-1})]_j$ to obtain an auxiliary point $u_{k,j}$ for all agents $j \in \cJ$. We let $u_k=[u_{k,1}, \ldots, u_{k,J}]^T$ be the stacked vector across all agents. In the next step, the iterate $x_{k-1,j}$ is updated using the mapping $[F(u_k)]_j$ to obtain the iterate $v_{k,j}$ for all $j \in \cJ$. Then, Algorithm~\ref{algo_proj_steps} is applied to reduce the infeasibility gap between $v_{k,j}$ and the set $X_j$ rendering  $x_{k,j}$ as the output at iteration $k$ for all $j \in \cJ$. The stacked vectors $v_k, x_k \in Y$ are as in~\eqref{stacked_vector}.
%
%
%
%
\begin{algorithm}
		\caption{Modified Korpelevich Method (For Agent $j$)}
		\label{algo_Kor_method}
		\begin{algorithmic}[1]
			\REQUIRE{$x_{0,j}$, $\a_k$}
                \FOR{$k=1,\ldots$}
                \STATE \textbf{Update:} $u_{k,j} = \Pi_{Y_j} \left[x_{k-1,j} - \a_{k-1} [F(x_{k-1})]_j\right]$
                \STATE \textbf{Update:} $v_{k,j} = \Pi_{Y_j} \left[x_{k-1,j} - \a_{k-1} [F(u_{k})]_j\right]$
                \STATE \textbf{Call Algorithm \ref{algo_proj_steps}:} Pass $v_{k,j}$ and $N_{k,j}$ and get $x_{k,j}$
                \ENDFOR
		\end{algorithmic}
\end{algorithm}	
%
%
%

  We next present a lemma for Algorithm~\ref{algo_Kor_method} showing a relation for the squared distance of the point $v_k$ from the solution $x^*$ of the VI$(S,F)$.
\begin{lemma}\label{Kor_lemma2}
    Under Assumptions~\ref{asum_closed_set},~\ref{asum_lipschitz}, and~\ref{asum_strong_monotone}, the following relation holds almost surely for the iterates $v_{k+1},x_k \in Y$ generated by Algorithm~\ref{algo_Kor_method}, with the stepsize $0<\a_k\le\frac{1}{4(L+\mu)}$, and the solution $x^*$,
    \begin{align*}
&\EXP{\| v_{k+1} - x^* \|^2 \mid \widetilde \cF_{k-1}} 
\leq  ( 1 - \a_k \mu) \| v_{k} - x^* \|^2 + \a_k^2 \left(2 + \frac{1}{\zeta_k} \right) \| F(x^*) \|^2 \nonumber\\
        & \ \ - \sum_{j=1}^J \left[( 1 - \a_k \mu)((1-q)^{-N_{k,j}}-1) - \zeta_k \right] \EXP{\dist^2(x_{k,j},S_j) \mid \cF_{k-1,j}}  \quad a.s., 
\end{align*}
    where $\zeta_k > 0$ is a scalar  and $q$ is given in~\eqref{quantity-q}.
\end{lemma}
\begin{proof}
    Using the definition of the iterate $v_{k+1}$ and the non-expansiveness property of the projection operator, we obtain
    \begin{align}
         \|v_{k+1} - x^*\|^2 &=  \| \Pi_{Y} [x_{k} - \alpha_k F(u_{k+1})] - x^* \|^2 \nonumber\\
         &\leq \| (x_{k} - x^*) - \alpha_k F(u_{k+1}) \|^2 - \| x_{k} - \alpha_k F(u_{k+1}) - v_{k+1} \|^2 \nonumber\\
         & = \| x_{k} - x^* \|^2 - \| v_{k+1} - x_{k} \|^2 + 2 \a_k \la F(u_{k+1}), x^* - v_{k+1} \ra . \label{Kor_bound2}
    \end{align}
    Next, we analyze the last term on the right hand side of \eqref{Kor_bound2} and write it as
    \begin{align}
        2 \a_k \la F(u_{k+1}), x^* \!-\! v_{k+1} \ra 
        \!=\! 2 \a_k \la F(u_{k+1}), x^* \!-\! u_{k+1} \ra + 2 \a_k \la F(u_{k+1}), u_{k+1} \!-\! v_{k+1} \ra . \label{Kor_bound3}
    \end{align}
    We can upper estimate the first term on the right hand side of \eqref{Kor_bound3} using strong monotonicity Assumption~\ref{asum_strong_monotone} on the mapping $F$, as follows:
    \begin{align}
        &2 \a_k \la F(u_{k+1}), x^* - u_{k+1} \ra \leq 2 \a_k \la F(x^*), x^* - u_{k+1} \ra - 2 \a_k \mu \| u_{k+1} - x^* \|^2 \nonumber\\
        &\!=\! 2 \a_k \la F(x^*), x^* - x_k \ra + 2 \a_k \la F(x^*), x_k - u_{k+1} \ra - 2 \a_k \mu \| u_{k+1} - x^* \|^2 \nonumber\\
        & \!\leq\! 2 \a_k \la F(x^*), x^* - x_k \ra + 2 \a_k^2 \|F(x^*)\|^2 + \frac{1}{2} \| x_k - u_{k+1} \|^2 - 2 \a_k \mu \| u_{k+1} - x^* \|^2 , \label{Kor_bound4}
    \end{align}
    where the last inequality is obtained by applying Young's inequality to the term $2 \a_k \la F(x^*), x_k - u_{k+1} \ra$. For the second term on the right hand side of \eqref{Kor_bound2} we write:
    \begin{align}
        \| v_{k+1} - x_{k} \|^2 = \| v_{k+1} - u_{k+1} \|^2 + \| x_{k} - u_{k+1} \|^2 - 2 \la x_{k} - u_{k+1}, v_{k+1} - u_{k+1} \ra . \label{Kor_bound6}
    \end{align}
    Substituting \eqref{Kor_bound4} in~\eqref{Kor_bound3}, and using the resulting relation and~\eqref{Kor_bound6} in~\eqref{Kor_bound2}, we obtain
    \begin{align}
        \|v_{k+1} - x^*\|^2 \leq \| x_{k} - x^* \|^2 - 2 \a_k \mu \| u_{k+1} - x^* \|^2 - \| v_{k+1} - u_{k+1} \|^2 - \frac{1}{2} \| x_{k} - u_{k+1} \|^2 \nonumber\\
        + 2 \la x_{k} \!-\! \a_k F(u_{k+1}) \!-\! u_{k+1}, v_{k+1} \!-\! u_{k+1} \ra + 2 \a_k \la F(x^*), x^* \!-\! x_k \ra + 2 \a_k^2 \|F(x^*)\|^2 . \label{Kor_bound7}
    \end{align}
    The fifth term on the right hand side of the preceding relation can be splitted as
    \begin{align}
        2 \la x_{k} - \a_k F(u_{k+1}) - u_{k+1}, v_{k+1} - u_{k+1} \ra = & 2 \la x_{k} - \a_k F(x_k) - u_{k+1}, v_{k+1} - u_{k+1} \ra \nonumber\\
        &+ 2 \la \a_k (F(x_k) - F(u_{k+1})), v_{k+1} - u_{k+1} \ra . \label{Kor_bound8}
    \end{align}
    By definition, $u_{k+1} = \Pi_Y[x_{k} - \a_k F(x_{k})]$. In addition, $v_{k+1} \in Y$. So, applying projection theorem on the first term of the right hand side of \eqref{Kor_bound8}, we have
    \begin{align}
        2 \la x_{k} - \a_k F(x_k) - u_{k+1}, v_{k+1} - u_{k+1} \ra \leq 0 . \label{Kor_bound9}
    \end{align}
    We bound the second term on the right hand side of~\eqref{Kor_bound8} using Cauchy-Schwarz inequality, Lipschitz continuity of the mapping $F$ (Assumption~\ref{asum_lipschitz}) and Young's inequality to obtain
    \begin{align}
        2 \la \a_k (F(x_k) - F(u_{k+1})), v_{k+1} - u_{k+1} \ra &\leq 2 \a_k \| F(x_k) - F(u_{k+1}) \| \| v_{k+1} - u_{k+1} \| \nonumber\\
        & \leq 2 \a_k L \| x_k - u_{k+1} \| \| v_{k+1} - u_{k+1} \| \nonumber\\
        & \leq 2 \a_k L \| x_k - u_{k+1} \|^2 + \frac{\a_k L}{2} \| v_{k+1} - u_{k+1} \|^2  . \label{Kor_bound10}
    \end{align}
    Using \eqref{Kor_bound8}, \eqref{Kor_bound9}, and \eqref{Kor_bound10}, from \eqref{Kor_bound7}
    we have that
    \begin{align}
        \|v_{k+1} - x^*\|^2 &\leq \| x_{k} - x^* \|^2 - 2 \a_k \mu \| u_{k+1} - x^* \|^2 - \left( 1-  \frac{\a_k L}{2} \right) \| v_{k+1} - u_{k+1} \|^2 \nonumber\\
        & \quad - \left( \frac{1}{2} - 2 \a_k L \right) \| x_{k} - u_{k+1} \|^2 - 2 \a_k \la F(x^*), x_k - x^* \ra + 2 \a_k^2 \|F(x^*)\|^2 . \label{Kor_bound11}
    \end{align}
    We consider the second term on the right hand side of \eqref{Kor_bound11}. We have
    \begin{align}
        \| x_k - x^* \|^2 = \| (x_k - u_{k+1}) + (u_{k+1} - x^*) \|^2 \leq 2 \| x_k - u_{k+1} \|^2 + 2 \| u_{k+1} - x^* \|^2, \nonumber
    \end{align}
    which provides an upper bound on the second term on the right hand side of \eqref{Kor_bound11} as
    \begin{align}
        - 2 \a_k \mu \| u_{k+1} - x^* \|^2 \le - \a_k \mu \| x_k - x^* \|^2 + 2 \a_k \mu \| x_k - u_{k+1} \|^2 . \label{Kor_bound12}
    \end{align}
    The fifth term on the right hand side of \eqref{Kor_bound11} can be  
    estimated using Young's inequality:
\begin{align*}
    - 2 \alpha_k \la F(x^*), x_k - x^* \ra \leq \frac{\a_k^2}{\zeta_k} \| F(x^*) \|^2 + \zeta_k \dist^2(x_k,S), 
\end{align*}
where $\zeta_k>0$ is an arbitrary scalar. 
     By combining the preceding inequality and~\eqref{Kor_bound12} with~\eqref{Kor_bound11}, we obtain the following relation
     \begin{align*}
        \|v_{k+1} - x^*\|^2 &\leq (1 - \a_k \mu) \| x_k - x^* \|^2 + \zeta_k \dist^2(x_k,S) + \left(2 + \frac{1}{\zeta_k} \right) \a_k^2 \|F(x^*)\|^2 \nonumber\\
        & - \left( 1- \frac{\a_k L}{2} \right) \| v_{k+1} - u_{k+1} \|^2 - \left( \frac{1}{2} - 2\a_k (L+\mu) \right) \| x_{k} - u_{k+1} \|^2.
    \end{align*}
Since $\frac{1}{4(L+\mu)} < \frac{1}{\mu}$ and $\frac{1}{4(L+\mu)} < \frac{2}{L}$, so with the step size $0 \leq \a_k \leq \frac{1}{4(L+\mu)}$, we have 
$1-\a_k \mu > 0$, $1- \frac{\a_k L}{2} >0$, and $\frac{1}{2} - 2\a_k (L+\mu) \geq 0$,
implying that
\begin{align}
        \|v_{k+1} - x^*\|^2 &\leq (1 - \a_k \mu) \| x_k - x^* \|^2 + \zeta_k \dist^2(x_k,S) + \left(2 + \frac{1}{\zeta_k} \right) \a_k^2 \|F(x^*)\|^2. 
        \label{eq-kor-end}
    \end{align}
Taking the conditional expectation in relation~\eqref{eq-kor-end}, given $\widetilde \cF_{k-1}$, and using~\eqref{dist_eq_opt} to simplify the first term on the right hand side of~\eqref{eq-kor-end}, we obtain the stated relation.  
\end{proof}
Lemma~\ref{Kor_lemma2} provides the same relation for the iterates of the modified Korpelevich method as that of Lemma~\ref{lem_proj} for the iterates of the projection method. The only difference in these lemmas is in the upper bound on the step size $\a_k$.
Therefore, we can use the same choice for $\zeta_k$ as in~\eqref{zeta_selection} for the projection method (Algorithm~\ref{algo_proj_method}),
which allows us to use any number $N_{k,j}\ge 1$ of samples for all agents $j$.

Next, we establish the almost sure and in-expectation convergence of the iterates $v_k, x_k \in Y$ to the solution $x^*$, and the convergence of the infeasibility gap $\dist(x_k,S)$ to $0$ almost surely and in expectation.
%
%

\begin{theorem}\label{thm_Kor_as_conv}
    Under Assumptions \ref{asum_closed_set},~\ref{asum-subgrad-bound},~\ref{asum_lipschitz},~\ref{asum_strong_monotone},~\ref{asum-regularmod}, and \ref{asum_step_size}, with a step-size $\a_k \leq \frac{1}{4(L+\mu)}$ for all $k$, $\zeta_k$ as in~\eqref{zeta_selection}, and any choice of the batch sizes $N_{k,j}\ge 1$ for all $k$ and $j$, the following statements hold for Algorithm~\ref{algo_Kor_method}:\\
    (i)~the iterates $v_{k}, x_{k} \in Y$ converge to $x^*$ almost surely and in expectation,\\
    (ii)~the distance $\dist(x_k,S)$ converges to $0$ almost surely and in expectation.
\end{theorem}
\begin{proof}
    The proof follows the same line of argument as that of Theorem~\ref{Proj_meth_thm_as_conv}, where we use Lemma~\ref{Kor_lemma2} instead of Lemma~\ref{lem_proj}. Additionally, here $\frac{1}{\zeta_k} \leq \frac{1-q}{q(1-\frac{\mu}{4(L+\mu)})}$, which along with Assumption~\ref{asum_step_size} makes $\sum_{k=0}^\infty \a_k^2/\zeta_k < \infty$.
\end{proof}

 Next, we establish convergence rate results in expectation.
\begin{theorem}\label{thm_Kor_conv_rates}
    Under the assumptions of Theorem~\ref{thm_Kor_as_conv} along with the step size selection as
    $\a_k \leq \min \left(\frac{2}{\mu(k+1)}, \frac{1}{4 (L+\mu)} \right)$ for all $k$, the expected norm squared distances of $v_{k+1}$ and $x_{k+1}$ from the solution $x^*$ are bounded above by the order
    \begin{align}
        O \left( \frac{(1+\kappa)^2}{T^2} + \frac{1}{\mu^2 T} \|F(x^*)\|^2 \right), \nonumber 
    \end{align}
    where $\kappa$ is the condition number defined in~\eqref{eq-cnum}. Moreover, the expected distance of $x_{T+1}$ to the set $S$ diminishes geometrically fast with the sample size, i.e.,
    \begin{align}
        \EXP{\dist(x_{T+1},S)} &\leq \max_{j \in \cJ} (1-q)^{\frac{N_{T+1,j}}{2}} \times O \left( \frac{(1+\kappa)}{T} + \frac{1}{\mu \sqrt{T}} \|F(x^*)\|^2 \right) , \nonumber
    \end{align}
    where $q$ is given by~\eqref{quantity-q}.
\end{theorem}
\begin{proof}
    The proof relies on Theorem~\ref{thm_Kor_as_conv} and follows the same line of arguments as that of Theorem~\ref{Proj_meth_thm_rate}. The only difference comes from different stepsize selection and defining  
    $\gamma_k=\a_k \mu$, which yields $\gamma_0=\frac{\mu}{4(L+\mu)}$ so that $\gamma_0^{-2}=16(1+\kappa^2)$.
\end{proof}
Theorem~\ref{thm_Kor_conv_rates} shows $O(1/T)$ rate of convergence for the iterates and a geometric rate of convergence for the infeasibility gap with respect to the number of selected samples times $O(1/\sqrt{T})$ for the iterate index $T+1$. The convergence rate results are of the same order as for the iterates generated by the projection method (cf.\ Theorem~\ref{Proj_meth_thm_rate}). 


\section{Modified Popov Method} \label{sec:Popov_method}
Here, we consider the Popov method \cite{popov1980modification} for solving \eqref{VI_problem}.
We modify the standard Popov method by employing feasibility updates (Algorithm~\ref{algo_proj_steps})  resulting in Algorithm~\ref{algo_Popov_method}. Similar to the Korpelevich method, Popov method uses two step updates using an auxiliary iterate $u_k$ and a single mapping computation at this iterate, rather than two computations as in the Korpelevich method.
%
%
%
%
\begin{algorithm}
		\caption{Modified Popov Method (For Agent $j$)}
		\label{algo_Popov_method}
		\begin{algorithmic}[1]
			\REQUIRE{$x_{0,j}$, $\a_k$}
                \FOR{$k=1,\ldots$}
                \STATE \textbf{Update:} $u_{k,j} = \Pi_{Y_j} \left[x_{k-1,j} - \a_{k-1} [F(u_{k-1})]_j\right]$
                \STATE \textbf{Update:} $v_{k,j} = \Pi_{Y_j} \left[x_{k-1,j} - \a_{k-1} [F(u_{k})]_j\right]$
                \STATE \textbf{Call Algorithm \ref{algo_proj_steps}:} Pass $v_{k,j}$ and $N_{k,j}$ and get $x_{k,j}$
                \ENDFOR
		\end{algorithmic}
\end{algorithm}	
%
%
Similar to the preceding algorithms, here also the same step size $\a_k$ is used across all the agents $j\in \cJ$ at any iteration $k \geq 1$, and Algorithm \ref{algo_proj_steps} is used to control the feasibility gap between $v_k$ and the constraint $X$, resulting in the output $x_k$. 
%
%
%
%
%

We start with a lemma providing a relation for the squared distance between the iterate $v_k \in Y$ and the solution  $x^*$ of the VI$(S,F)$.
\begin{lemma}\label{Popov_lem2}
    Under Assumptions~\ref{asum_closed_set},~\ref{asum_lipschitz}, and~\ref{asum_strong_monotone}, the following holds almost surely for the iterates $v_{k+1},x_k \in Y$ of Algorithm~\ref{algo_Popov_method} and the solution $x^*$,
        \begin{align*}
        &\EXP{\| v_{k+1} - x^* \|^2 \mid \widetilde \cF_{k-1}} \leq (1-\mu \a_k + 2 \tau L^2 \a_k^2) \| v_k - x^* \|^2 \nonumber\\
        &+ \a_k^2 \left(\nu + \frac{1}{\zeta_k} \right) \|F(x^*)\|^2 - \left(1-\frac{2}{\tau}-\frac{2\a_k L}{\tau}\right) \EXP{\| v_{k+1} - u_{k+1} \|^2 \mid \widetilde \cF_{k-1}} \nonumber\\
        & + \a_k L \tau \| v_k - u_k \|^2 - \left( 1 - \frac{1}{\nu} - 2\a_k\mu - \a_k L \tau \right) \EXP{\| x_{k} - u_{k+1} \|^2 \mid \widetilde \cF_{k-1}} \nonumber\\
        & -\sum_{j=1}^J \left[(1 \!-\! \mu \a_k \!+\! 2 \tau L^2 \a_k^2)((1-q)^{-N_{k,j}}-1) - \zeta_k \right] \EXP{\dist^2(x_{k,j},S_j) \!\mid \! \cF_{k-1,j}}, 
    \end{align*}
    where $\zeta_k > 0$, $\tau >0$, and $\nu>0$ are arbitrary scalars, and $q$ is given by \eqref{quantity-q}.
\end{lemma}
\begin{proof}
    Using the non-expansiveness property of the projection on the definition of $v_{k+1}$ from Algorithm~\ref{algo_Popov_method}, we obtain the following relation
    \begin{align}
        \| v_{k+1} - x^* \|^2 \leq \| x_{k} - x^* \|^2 - \| v_{k+1} - x_{k} \|^2 + 2 \a_k \la F(u_{k+1}), x^* - v_{k+1} \ra . \label{Pop_bound2}
    \end{align}
    The second term in~\eqref{Pop_bound2} can be written as
    \begin{align}
        \| v_{k+1} - x_{k} \|^2 = \| v_{k+1} - u_{k+1} \|^2 + \| x_{k} - u_{k+1} \|^2 - 2 \la x_{k} - u_{k+1}, v_{k+1} - u_{k+1} \ra . \label{pop_secterm}
    \end{align}
    The third term on the right hand of \eqref{Pop_bound2} can be written as
    \begin{align}
        2 \a_k \la F(u_{k+1}), x^* - v_{k+1} \ra = 2 \a_k \la F(u_{k+1}), x^* - u_{k+1} \ra + 2 \a_k \la F(u_{k+1}), u_{k+1} - v_{k+1} \ra . \label{pop_map}
    \end{align}
    By substituting the relations~\eqref{pop_secterm} and~\eqref{pop_map} back in relation~\eqref{Pop_bound2}, we obtain
    \begin{align}
        \| v_{k+1} - x^* \|^2 &\leq \| x_{k} - x^* \|^2 - \| v_{k+1} - u_{k+1} \|^2 - \| x_{k} - u_{k+1} \|^2 \nonumber\\
        &+ 2 \la x_{k} - \a_k F(u_{k+1}) - u_{k+1}, v_{k+1} - u_{k+1} \ra + 2 \a_k \la F(u_{k+1}), x^* - u_{k+1} \ra. \label{Pop_no1}
    \end{align}
    The last term on the right hand side of~\eqref{Pop_no1} can be upper bounded using the strong monotonicity of the mapping $F(\cdot)$ (Assumption~\ref{asum_strong_monotone}), as follows:
    \begin{align}
        2 \a_k \la F(u_{k+1}), x^* - u_{k+1} \ra &\leq 2 \a_k \la F(x^*), x^* - u_{k+1} \ra - 2 \a_k \mu \| u_{k+1} - x^* \|^2 \nonumber\\
        = 2 \a_k \la F(x^*), x^* - x_k \ra &+ 2 \a_k \la F(x^*), x_k - u_{k+1} \ra - 2 \a_k \mu \| u_{k+1} - x^* \|^2 \nonumber\\
        \leq 2 \a_k \la F(x^*), x^* - x_k \ra &+ \a_k^2 \nu \|F(x^*)\|^2 + \frac{1}{\nu} \| x_k - u_{k+1} \|^2 - 2 \a_k \mu \| u_{k+1} - x^* \|^2 , \label{Pop_bound4}
    \end{align}
    where the last inequality in~\eqref{Pop_bound4} is obtained by applying Young's inequality to the term $2 \a_k \la F(x^*), x_k - u_{k+1} \ra$, with some $\nu>0$. The first term on the right hand side of \eqref{Pop_bound4} can be estimated using Young's inequality, again, and obtain
    \begin{align}
        2 \alpha_k \la F(x^*), x^* - x_k \ra \leq \frac{\a_k^2}{\zeta_k} \| F(x^*) \|^2 + \zeta_k \dist^2(x_k,S) , \label{Pop_no2}
    \end{align}
    where $\zeta_k > 0$ is a scalar. For the last term on the right hand side of \eqref{Pop_bound4}, we have 
    \begin{align}
    - 2 \a_k\mu \|u_{k+1}-x^* \|^2 \le - \a_k \mu \|x_k-x^*\|^2 + 2 \a_k\mu \|u_{k+1}-x_k\|^2 . \label{Pop_bound_st_mon_term}
    \end{align}
    Combining the estimates in~\eqref{Pop_no2} and~\eqref{Pop_bound4} with relation~\eqref{Pop_bound4},
     we obtain
     \begin{align}\label{eq-pop-1}
        &2 \a_k \la F(u_{k+1}), x^* - u_{k+1} \ra \leq \frac{\a_k^2}{\zeta_k} \| F(x^*) \|^2 + \zeta_k \dist^2(x_k,S)
        + \a_k^2 \nu \|F(x^*)\|^2\cr 
        &+ \frac{1}{\nu} \| x_k - u_{k+1} \|^2 
        - \a_k \mu \|x_k-x^*\|^2 + 2 \a_k\mu \|u_{k+1}-x_k\|^2 .\end{align}

    Next, the fourth term on the right hand side of~\eqref{Pop_no1} can be written as
    \begin{align}
        2 \la x_{k} - \a_k F(u_{k+1}) - u_{k+1}, v_{k+1} - u_{k+1} \ra = & 2 \la x_{k} - \a_k F(u_k) - u_{k+1}, v_{k+1} - u_{k+1} \ra \nonumber\\
        &+ 2 \la \a_k (F(u_k) - F(u_{k+1})), v_{k+1} - u_{k+1} \ra . \label{Pop_bound8}
    \end{align}
    By definition, $u_{k+1} = \Pi_Y [x_{k} - \a_k F(u_k)]$. Also $v_{k+1} \in Y$. Thus, by the Projection Theorem for the first term on the right hand side of \eqref{Pop_bound8}, we obtain
    \begin{align}
        2 \la x_{k} - \a_k F(u_k) - u_{k+1}, v_{k+1} - u_{k+1} \ra \leq 0 , \label{Pop_bound9}
    \end{align}
    and hence we can drop this term from~\eqref{Pop_bound8}. We estimate the second term on the right hand side of~\eqref{Pop_bound8} by using Cauchy-Schwarz inequality, Lipschitz continuity of $F$ (Assumption~\ref{asum_lipschitz}), and Young's inequality, to obtain
    \begin{align}
        &2 \la \a_k (F(u_k) - F(u_{k+1})), v_{k+1} - u_{k+1} \ra \leq 2 \a_k \| F(u_k) - F(u_{k+1}) \| \| v_{k+1} - u_{k+1} \| \nonumber\\
        & \leq 2 \a_k L \| u_k - u_{k+1} \| \| v_{k+1} - u_{k+1} \| \nonumber\\
        & \leq 2 \a_k L (\| u_k - y \| + \| x_k - y\| + \| x_k - u_{k+1} \|) \| v_{k+1} - u_{k+1} \| \nonumber\\
        & \leq 2 \a_k L (\| u_k - v_k \| + \| v_k - y \| + \| x_k - y\| + \| x_k - u_{k+1} \|) \| v_{k+1} - u_{k+1} \| \nonumber\\
        &\leq \a_k L \tau (\| v_k - u_k \|^2+\| x_k - u_{k+1} \|^2) + \a_k^2 L^2 \tau (\| v_k - y \|^2 + \| x_k - y\|^2) \nonumber\\
        & \hspace{6cm}+ \left(\frac{2\a_k L}{\tau}+\frac{2}{\tau}\right) \| v_{k+1} - u_{k+1} \|^2,  \label{Pop_bound10}
    \end{align}
    where $\tau>0$ is arbitrary. Using the estimates~\eqref{eq-pop-1} and~\eqref{Pop_bound10} in~\eqref{Pop_no1}, we obtain 
    \begin{align}\label{eq-pop-main}
        \| v_{k+1} - x^* \|^2 &\leq (1-\a_k\mu + \a_k^2 L^2 \tau) \| x_{k} - x^* \|^2 + \a_k^2 L^2 \tau \| v_k - x^* \|^2 + \zeta_k \dist^2(x_k,S)  \nonumber\\
        & + \a_k^2 \left(\nu + \frac{1}{\zeta_k} \right) \|F(x^*)\|^2 - \left( 1 - \frac{1}{\nu} - \a_k(2\mu + L \tau) \right) \| x_{k} - u_{k+1} \|^2 \nonumber\\
        & + \a_k L \tau \| v_k - u_k \|^2  - \left( 1 - \frac{2}{\tau} - \frac{2\a_k L}{\tau} \right) \| v_{k+1} - u_{k+1} \|^2 . \nonumber
    \end{align}
Upon taking the conditional expectation in relation~\eqref{eq-pop-main}, given $\widetilde \cF_{k-1}$, and using the relation~\eqref{dist_eq_opt}, we obtain the stated result. 
\end{proof}
%
%
%
The coefficient of the infeasibility term $\EXP{\dist^2(x_{k,j},S_j) \mid \cF_{k-1,j}}$ in Lemma~\ref{Popov_lem2} can be controlled by choosing $\zeta_k$ so that $(1-\mu \a_k + 2 \tau L^2 \a_k^2)((1-q)^{-N_{k,j}}-1) \geq \zeta_k$, or equivalently 
    \begin{align}
        N_{k,j} \geq \frac{\log\left(\frac{ \zeta_k+1-\a_k\mu  + 2 \tau L^2 \a_k^2}{1-\a_k\mu + 2 \tau L^2 \a_k^2}\right)}{\log\left(\frac{1}{1-q}\right)} = \log_{\frac{1}{1-q}}\left(\frac{ \zeta_k+1-\a_k\mu  + 2 \tau L^2 \a_k^2}{1-\a_k\mu + 2 \tau L^2 \a_k^2}\right) , \; \forall k\geq 1 \;, j \in \cJ. \nonumber
    \end{align}
We want Algorithm \ref{algo_Popov_method} to run for to work for any number $N_{k,j}$ of samples, i.e., $N_{k,j} \geq 1$ for all $k \geq 1$ and $j \in \cJ$. Hence, we impose the condition that $\frac{ \zeta_k+1-\a_k\mu  + 2 \tau L^2 \a_k^2}{1-\a_k\mu + 2 \tau L^2 \a_k^2} = \frac{1}{1-q}$ resulting in the following choice for $\zeta_k$:
\begin{align}
    \zeta_k = \frac{q}{1-q} (1-\a_k\mu + 2 \tau L^2 \a_k^2) . \label{zeta_selection4}
\end{align}
We will have $\zeta_k>0$ when $1-\a_k \mu > 0$. 
We have the following result regarding the convergence of the modified Popov method.

%
%


\begin{theorem}\label{Popov_thm1}
    Let Assumptions \ref{asum_closed_set},~\ref{asum-subgrad-bound},~\ref{asum_lipschitz},~\ref{asum_strong_monotone},~\ref{asum-regularmod}, and \ref{asum_step_size} hold, with $\zeta_k$ as~\eqref{zeta_selection4}, and the constants $\tau > 2$ and $\nu>1$. 
    Also, let the step-sizes be non-increasing, i.e., $\a_{k+1} \le \a_{k}$, and satisfy $\a_k \leq \min \left( \frac{\mu}{4 \tau L^2}, \frac{1-\frac{2}{\tau}}{\frac{2L}{\tau}+L\tau}, \frac{1-\frac{1}{\nu}}{2\mu+L\tau} \right)$ for all $k\ge 0$. Then, the following results hold for the modified Popov method:\\
    (i)~the iterates $x_k,v_k\in Y$ converge to $x^*$ almost surely and in expectation,\\
    (ii)~the distance $\dist(x_k,S)$ converges to $0$ almost surely and in expectation.
\end{theorem}
\begin{proof}
For the step size  $\a_k \leq \frac{\mu}{4 \tau L^2},$ with $\tau>2$,
we have that $\a_k<\frac{\mu}{8L^2}\le \frac{1}{8\mu}$ since $\mu\le L$. Hence, $\a_k \mu < 1$, and
    with the selection of $\zeta_k$ as in~\eqref{zeta_selection4}, 
    from Lemma~\ref{Popov_lem2} we obtain that 
    almost surely
    \begin{align}
        &\EXP{\| v_{k+1} - x^* \|^2 \mid \widetilde \cF_{k-1}} \leq (1-\mu \a_k + 2 \tau L^2 \a_k^2) \| v_k - x^* \|^2 \nonumber\\
        &+ \a_k^2 \left(\nu + \frac{1}{\zeta_k} \right) \|F(x^*)\|^2 - \left(1-\frac{2}{\tau}-\frac{2\a_k L}{\tau}\right) \EXP{\| v_{k+1} - u_{k+1} \|^2 \mid \widetilde \cF_{k-1}} \nonumber\\
        & + \a_k L \tau \| v_k - u_k \|^2 - \left( 1 - \frac{1}{\nu} - 2\a_k\mu - \a_k L \tau \right) \EXP{\| x_{k} - u_{k+1} \|^2 \mid \widetilde \cF_{k-1}}, \label{Popov_no3}
    \end{align}
    where we  drop the last term on the right hand side of the relation in Lemma~\ref{Popov_lem2}.
    
    Let us define two sequences
    \begin{align}
        h_k &= \| v_k - x^* \|^2 + \a_k L \tau \| v_k - u_k \|^2, \nonumber\\
        f_k &= \left(1-\frac{2}{\tau}-\frac{2\a_k L}{\tau} - \a_{k+1}  L \tau \right) \EXP{\| v_{k+1} - u_{k+1} \|^2 \mid \widetilde \cF_{k-1}} \nonumber\\
        &+ \left( 1 - \frac{1}{\nu} - 2\a_k\mu - \a_kL \tau \right) \EXP{\| x_{k} - u_{k+1} \|^2 \mid \widetilde \cF_{k-1}} . \label{Popov_thm1_eq3}
    \end{align}
    Then, using these sequences 
    relation~\eqref{Popov_no3} can be written as: almost surely,
    \begin{align}
        \EXP{h_{k+1} \mid \widetilde \cF_{k-1}} \leq \max(1-\mu \a_k + 2 \tau L^2 \a_k^2, \a_k L \tau) h_k - f_k + \a_k^2 \left(\nu + \frac{1}{\zeta_k} \right) \|F(x^*)\|^2. \label{sup_mart}
    \end{align}
    We want $1-\mu \a_k + 2 \tau L^2 \a_k^2 \leq 1 - \frac{\mu \a_k}{2}$, $\a_k L \tau \leq 1$, $1-\frac{2}{\tau}-\frac{2\a_k L}{\tau} - \a_{k+1}  L \tau \geq 0$, and $1 - \frac{1}{\nu} - 2\a_k\mu - \a_k L \tau \geq 0$, which together with $\a_{k+1}\le \a_k$ for all $k$, gives the step size selection as $0< \a_k \leq \min \left( \frac{\mu}{4 \tau L^2}, \frac{1}{L\tau}, \frac{1-\frac{2}{\tau}}{\frac{2L}{\tau}+L\tau}, \frac{1-\frac{1}{\nu}}{2\mu+L\tau} \right)$. Note that $\frac{1}{L \tau} > \min \left(\frac{1-\frac{2}{\tau}}{\frac{2L}{\tau}+L\tau}, \frac{1-\frac{1}{\nu}}{2\mu+L\tau} \right)$ and, hence, that term can be ignored. This leaves us with the step size selection rule as
    \begin{align}
         0< \a_k \leq \min \left( \frac{\mu}{4 \tau L^2}, \frac{1-\frac{2}{\tau}}{\frac{2L}{\tau}+L\tau}, \frac{1-\frac{1}{\nu}}{2\mu+L\tau} \right), \label{Popov-stepsize2}
    \end{align}
    which is assumed.
    With this step size, from~\eqref{sup_mart} we conclude that the following relation holds almost surely
    \begin{align}
        \EXP{h_{k+1} \mid \widetilde \cF_{k-1}} \leq \max \left(1 - \frac{\mu \a_k}{2}, \a_k L \tau \right) h_k - f_k + \a_k^2 \left(\nu + \frac{1}{\zeta_k} \right) \|F(x^*)\|^2 , \label{sup_mart2}
    \end{align}
    with $f_k\ge0$.
    Since $\a_k $ is diminishing to 0 (due to the assumption that $\sum_{k=0}^\infty \a_k^2<\infty$), after some iteration index $N$, we will have  $\max \left(1 - \frac{\mu \a_k}{2}, \a_k L \tau \right) = 1 - \frac{\mu \a_k}{2}$ for all $k\ge N$. 
    Hence,  \eqref{sup_mart2} implies that for all  $k \geq N$, 
    \begin{align}
        \EXP{h_{k+1} \mid \widetilde \cF_{k-1}} \leq \left(1 - \frac{\mu \a_k}{2} \right) h_k + \a_k^2 \left(\nu + \frac{1}{\zeta_k} \right) \|F(x^*)\|^2 \qquad a.s., \label{Popov-mart}
    \end{align}
     where we omit $f_k$ from  \eqref{sup_mart2} since $f_k\ge0$. We want to upper estimate the quantity $1/\zeta_k$ in the last term of~\eqref{Popov-mart}. Note that for $\a_k$ satisfying~\eqref{Popov-stepsize2}, we have $1-\a_k\mu > 0$. Using the selection of $\zeta_k$ as in~\eqref{zeta_selection4}, we get
    \begin{align}
        O \left( \frac{1}{\zeta_k} \right) \leq O \left(\frac{1}{((1- \max_k \{\a_k\}\mu) + 2 \tau L^2 \a_k^2)} \right) = O(1), \label{zeta_order}
    \end{align}
    Also $\nu = O(1)$. Hence, by Assumption~\ref{asum_step_size}, we obtain $\sum_{k=0}^\infty \a_k^2 (\nu+1/\zeta_k) < \infty$. Use of the preceding relations along with Lemma \ref{lem_near_super_martingale} (with the index-shift) on the relation \eqref{Popov-mart}, we conclude that almost surely $\lim_{k \rightarrow \infty} h_k = 0$ and $\lim_{k \rightarrow \infty} \EXP{h_k} = 0$. Using the definition of $h_k$ from \eqref{Popov_thm1_eq3}, we can conclude that almost surely $\lim_{k \rightarrow \infty} \|v_k - x^* \|^2 = 0$ and $\lim_{k \rightarrow \infty} \EXP{\|v_k - x^*\|^2} = 0$. The rest of the proof follows the same line of analysis as in the proof of Theorem~\ref{Proj_meth_thm_as_conv} from relation \eqref{mart_proj1} onwards.
\end{proof}

Next, we state the convergence rate results for Algorithm~\ref{algo_Popov_method}.
\begin{theorem}\label{Popov_thm2}
Under the assumptions of Theorem~\ref{Popov_thm1}, with non-increasing step sizes $\a_k \leq \min \left( \frac{4}{\mu(k+1)}, \frac{\mu}{4 \tau L^2}, \frac{1-\frac{2}{\tau}}{\frac{2L}{\tau}+L\tau}, \frac{1-\frac{1}{\nu}}{2\mu+L\tau} \right)$ for all $k$, with the constants $\tau > 2$ and $\nu > 1$, the quantities $\EXP{\| v_{T+1} - x^* \|^2}$ and $\EXP{\| x_{T+1} - x^* \|^2}$ are bounded above by
    \begin{align}
        O \left(\frac{4}{\mu^2 \a_0^2 T^2} + \frac{4\tau \kappa}{\mu \a_1 T^2} + \frac{1}{\mu^2 T} \|F(x^*)\|^2 \right) . \nonumber
    \end{align}
    Moreover, the expected distance of $x_k$ to the set $S$ satisfies
    \begin{align}
        &\EXP{\dist(x_{T+1},S)} \leq \max_{j \in \cJ} (1-q)^{\frac{N_{T+1,j}}{2}} \times O \left(\frac{2}{\mu \a_0 T} + \frac{2 \sqrt{\tau \kappa}}{\sqrt{\mu \a_1}\, T} + \frac{1}{\mu \sqrt{T}} \|F(x^*)\| \right) , \nonumber
    \end{align}
    where $q$ is a constant given in \eqref{quantity-q}.
\end{theorem}
\begin{proof}
    Since the conditions of 
    Theorem~\ref{Popov_thm1} are satisfied with $\zeta_k$ as in~\eqref{zeta_selection4},
    we obtain relation~\eqref{Popov_no3}, which we restate below for convenience: {a.s.}
    \begin{align}
        &\EXP{\| v_{k+1} - x^* \|^2 \mid \widetilde \cF_{k-1}} \leq (1-\mu \a_k + 2 \tau L^2 \a_k^2) \| v_k - x^* \|^2 \nonumber\\
        &+ \a_k^2 \left(\nu + \frac{1}{\zeta_k} \right) \|F(x^*)\|^2 - \left(1-\frac{2}{\tau}-\frac{2\a_k L}{\tau}\right) \EXP{\| v_{k+1} - u_{k+1} \|^2 \mid \widetilde \cF_{k-1}} \nonumber\\
        & + \a_k L \tau \| v_k - u_k \|^2 - \left( 1 - \frac{1}{\nu} - 2\a_k\mu - \a_k L \tau \right) \EXP{\| x_{k} - u_{k+1} \|^2 \mid \widetilde \cF_{k-1}}, \label{Popov_no3a}
    \end{align}
    where the constants $\tau>2$, and $\nu>1$. We select the step size $\a_k$ as
    \begin{align}
         0 < \a_k \leq \min \left( \frac{4}{\mu(k+1)}, \frac{\mu}{4 \tau L^2}, \frac{1-\frac{2}{\tau}}{\frac{2L}{\tau}+L\tau}, \frac{1-\frac{1}{\nu}}{2\mu+L\tau} \right). \label{Popov_stepsize2}
    \end{align}
    Since $\a_k\le \frac{\mu}{4\tau L^2}$, it follows that $2\tau L^2\a_k\le \frac{\mu}{2}$,
    implying that $1-\mu \a_k + 2 \tau L^2 \a_k^2 \leq 1 - \frac{\mu \a_k}{2}$. The condition
    $\a_k\le \frac{1-\frac{1}{\nu}}{2\mu+L\tau}$
    implies that $1 - \frac{1}{\nu} - 2\a_k\mu - \a_k L \tau\ge0$.
    Thus, using these estimates in~\eqref{Popov_no3a} and taking the total expectation in the resulting relation, we obtain
    \begin{align}
        \EXP{\| v_{k+1} - x^* \|^2} \leq &\left(1- \frac{\mu \a_k}{2} \right) \EXP{\| v_k - x^* \|^2} - \left(1-\frac{2}{\tau}-\frac{2\a_k L}{\tau}\right) \EXP{\| v_{k+1} - u_{k+1} \|^2} \nonumber\\
        & + \a_k L \tau \EXP{\| v_k - u_k \|^2} + \a_k^2 \left(\nu + \frac{1}{\zeta_k} \right) \|F(x^*)\|^2 . \label{Popov_thm2_eqn1}
    \end{align}
    Let us define $\gamma_k=\frac{\mu \a_k}{2}$, 
    which in view of the condition on $\a_k$ in~\eqref{Popov_stepsize2} yields
    \begin{align}
        \gamma_k \leq \min \left( \frac{2}{k+1}, \frac{\mu^2}{8 \tau L^2}, \frac{\mu \left(1- \frac{2}{\tau}\right)}{\frac{4L}{\tau}+ 2L\tau}, \frac{\mu \left(1 - \frac{1}{\nu}\right)}{4\mu + 2L\tau} \! \right) , \label{Popov_step_size}
    \end{align}
    with $\tau > 2$ and $\nu>1$. Then, we divide both sides of relation~\eqref{Popov_step_size} by $\gamma_k^2$ and use the fact that $\frac{1-\gamma_k}{\gamma_k^2} \leq \frac{1}{\gamma_{k-1}^2}$ to obtain
    \begin{align}
        \gamma_k^{-2} \EXP{\| v_{k+1} - x^* \|^2} \leq &\gamma_{k-1}^{-2} \EXP{\| v_k - x^* \|^2} 
        \!-\! \left(\frac{1}{\gamma_k^2}\left(1 \!-\! \frac{2}{\tau}\right)-\frac{4 L}{\mu \tau \gamma_k}\right) \EXP{\| v_{k+1} \!-\! u_{k+1} \|^2} \nonumber\\
        & + \frac{2 L \tau}{\mu \gamma_k} \EXP{\| v_k - u_k \|^2} + \frac{4}{\mu^2} \left(\nu + \frac{1}{\zeta_k} \right) \|F(x^*)\|^2 . \label{Popov_no5}
    \end{align}
    Next, we sum across all the terms of equation \eqref{Popov_no5} from $k=1$ to some iterate index $T\ge1$ and use the relation \eqref{zeta_order} and $\nu = O(1)$ to obtain
    \begin{align}
        \gamma_T^{-2} &\EXP{\| v_{T+1} - x^* \|^2} + \sum_{k=2}^{T} \left(\frac{1}{\gamma_{k-1}^2}\left(1-\frac{2}{\tau}\right)-\frac{4 L}{\mu \tau \gamma_{k-1}} - \frac{2 L \tau}{\mu \gamma_k} \right) \EXP{\| v_{k} - u_{k} \|^2} + \nonumber\\
        &+\left(\frac{1}{\gamma_T^2}\left(1-\frac{2}{\tau}\right)-\frac{4 L}{\mu \tau \gamma_T}\right) \EXP{\| v_{T+1} - u_{T+1} \|^2} \nonumber\\
        &\leq \gamma_{0}^{-2} \EXP{\| v_1 - x^* \|^2} + \frac{2 L \tau}{\mu \gamma_1} \EXP{\| v_1 - u_1 \|^2} + O \left( \frac{4 T}{\mu^2} \right) \|F(x^*)\|^2 . \label{Popov_thm2_eqn3}
    \end{align}
    The condition on the step size in~\eqref{Popov_step_size} ensures that all the terms on the left hand side and the right hand side of~\eqref{Popov_thm2_eqn3} are non-negative. Hence, we can drop the second, and third terms on the left hand side of~\eqref{Popov_thm2_eqn3} and, thus, obtain 
    \begin{align}
        \gamma_T^{-2} \EXP{\| v_{T+1} - x^* \|^2} \!\leq\! \frac{4}{\mu^2 \a_0^2} \EXP{\| v_1 - x^* \|^2} \!+\! \frac{4 L \tau}{\mu^2 \a_1} \EXP{\| v_1 - u_1 \|^2} \!+\! O \left( \frac{4 T}{\mu^2} \right) \|F(x^*)\|^2 . \nonumber
    \end{align}
    Note that $\gamma_T^{-2} = \frac{(T+1)^2}{4}$ when $T$ is sufficiently large. The preceding relation gives an upper bound for $\EXP{\| v_{T+1} - x^* \|^2}$. Taking the total expectation on both sides of~\eqref{dist_eq_opt}, we have $\EXP{\| x_{T+1} - x^* \|^2} \leq \EXP{\| v_{T+1} - x^* \|^2}$, which gives an upper bound for $\EXP{\| x_{T+1} - x^* \|^2}$.
    Now, consider~\eqref{dist_eq_geom} for $k=T+1$ and take the total expectation on both sides of the equation to obtain
    \begin{align}
        \EXP{\dist^2(x_{T+1},S)} &\leq \EXP{\|v_{T+1}-x^*\|^2} \max_{j \in \cJ} (1-q)^{N_{k,j}} \nonumber\\
        &\leq \max_{j \in \cJ} (1-q)^{N_{T+1,j}} \times O \left(\frac{4}{\mu^2 \a_0^2 T^2} + \frac{4\tau \kappa}{\mu \a_1 T^2} + \frac{1}{\mu^2 T} \|F(x^*)\|^2 \right) . \nonumber
    \end{align}
    Using Jensen's inequality, $\EXP{\dist(x_{T+1},S)} \leq \sqrt{\EXP{\dist^2(x_{T+1},S)}}$, we finally obtain the estimate for the expected feasibility violation of $x_k$. 
\end{proof}
Theorem \ref{Popov_thm2} shows that the iterates $v_{T+1}$ and $x_{T+1}$, for Algorithm~\ref{algo_Popov_method}, converge to the solution $x^*$ in expectation with the same rate $O(1/T)$ as for Algorithms~\ref{algo_proj_method} and \ref{algo_Kor_method}. Also, the convergence rate of $\EXP{\dist(x_{T+1},S)}$ for all the Algorithms~\ref{algo_proj_method}, \ref{algo_Kor_method}, and \ref{algo_Popov_method} is the same.
The rates and the step size in Theorem~\ref{Popov_thm2} depend on parameters $\tau>2$ and $\nu>1$, which are artifacts of our analysis.
%
We can choose $\nu =2$ so as to have similar coefficients for $1/T$ term for all the algorithms previously discussed. For the selection of $\tau$, from relation~\eqref{Popov_stepsize2} we see that the second and the fourth terms ensure a larger step size for a small $\tau$, whereas the third term has inflection points for $\tau$, which when solved gives $\tau = 2 + \sqrt{6}$. This choise of $\tau$ might be good for the third term but not for the second term in \eqref{Popov_stepsize2}. Optionally, we can choose $\tau$ so that $ \frac{\mu}{4 \tau L^2} = \frac{1-\frac{2}{\tau}}{\frac{2L}{\tau}+L\tau}$, which gives $\tau = \frac{8L^2(1+\sqrt{1+\mu(4/L-\mu/L^2)})}{8L^2 - 2\mu L} = \frac{8(1+\sqrt{1+(4/\kappa - 1/\kappa^2)})}{8 - 2/\kappa}$. Since $\kappa \ge 1$, we have $4/\kappa > 1/\kappa^2$ and $2/\kappa < 2 < 8$. Hence, 
$8(1+\sqrt{1+(4/\kappa - 1/\kappa^2)})>16$ and 
$8 - 2/\kappa < 8$, implying that $\tau>2$.



\section{Simulation}\label{sec:simulations}

We present the experimental demonstration of our algorithms on a two player matrix game and an imitation learning toy example. For all the simulations, we took step sizes as per Theorems \ref{Proj_meth_thm_rate}, \ref{thm_Kor_conv_rates}, and \ref{Popov_thm2}, with $\nu =2$ and $\tau = \frac{8L^2(1+\sqrt{1+\mu(4/L-\mu/L^2)})}{8L^2 - 2\mu L}$.

\begin{figure*}[t]\centering
	\begin{subfigure}{0.24\textwidth}
		\includegraphics[width=\linewidth, height = 0.75\linewidth]
		{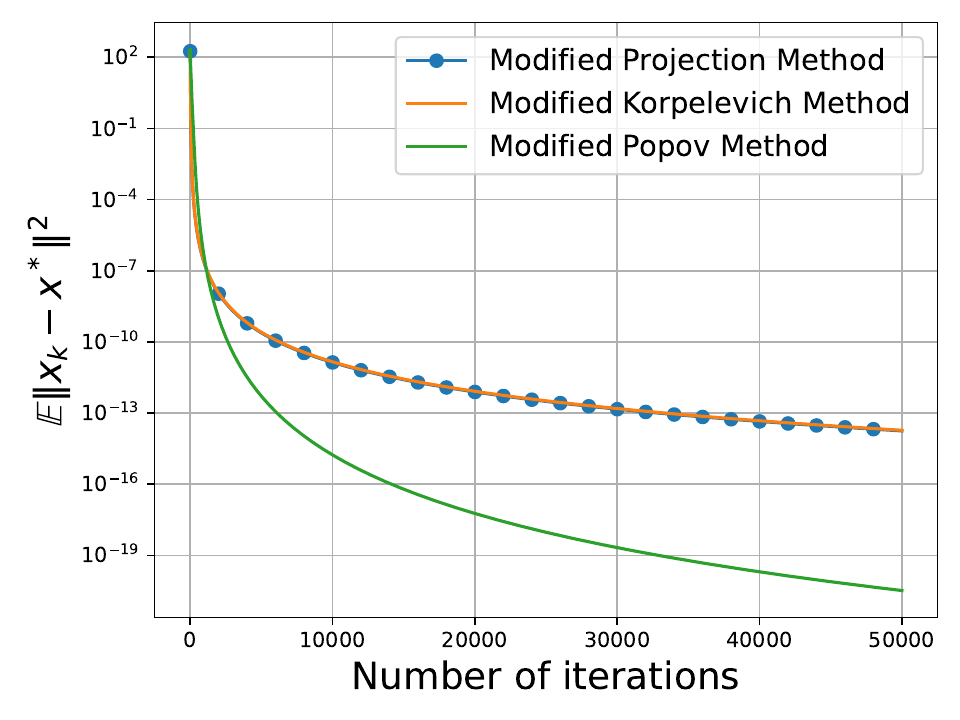}
		\caption{$\mu=1$ and $L=3$}
		\label{well_cond}
	\end{subfigure}
	\begin{subfigure}{0.24\textwidth}
		\includegraphics[width=\linewidth,height = 0.75\linewidth]
		{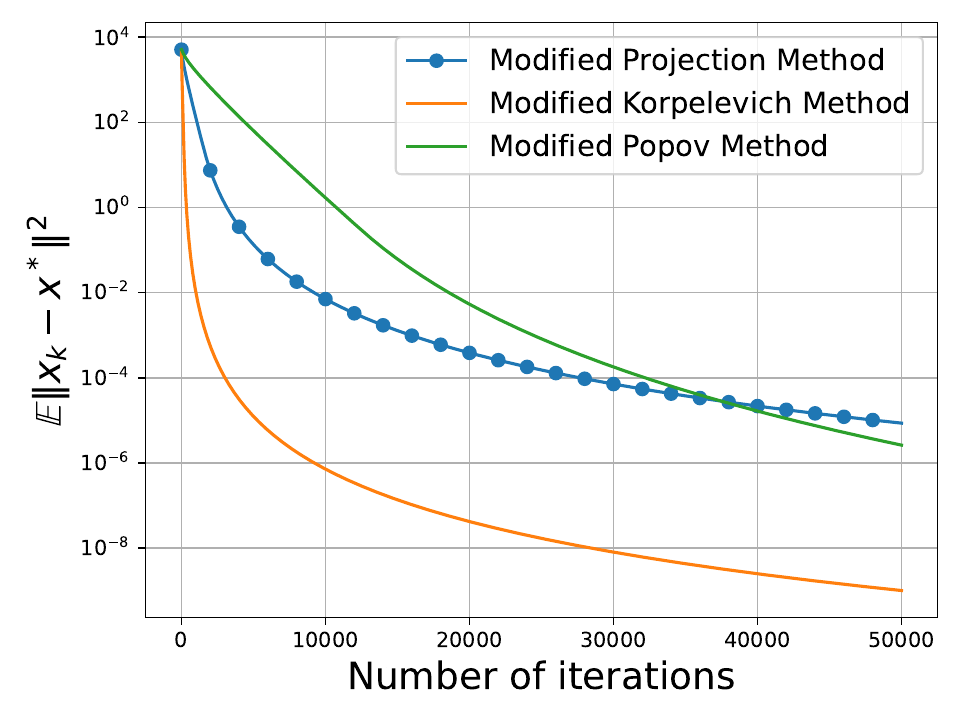}
		\caption{$\mu=0.05$, $L=1$}
		\label{ill_cond}
	\end{subfigure}
	\begin{subfigure}{0.24\textwidth}
		\includegraphics[width=\linewidth,height = 0.75\linewidth]
		{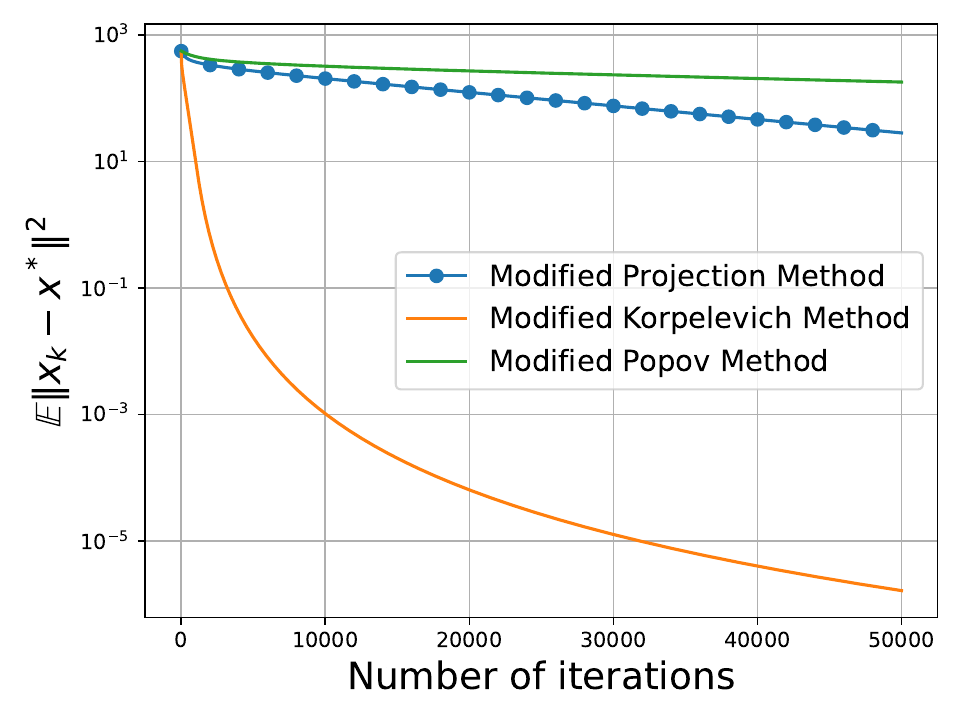}
		\caption{$\mu=0.01$, $L=10$}
		\label{ill_cond2}
	\end{subfigure}
        \begin{subfigure}{0.24\textwidth}
		\includegraphics[width=\linewidth,height = 0.75\linewidth]{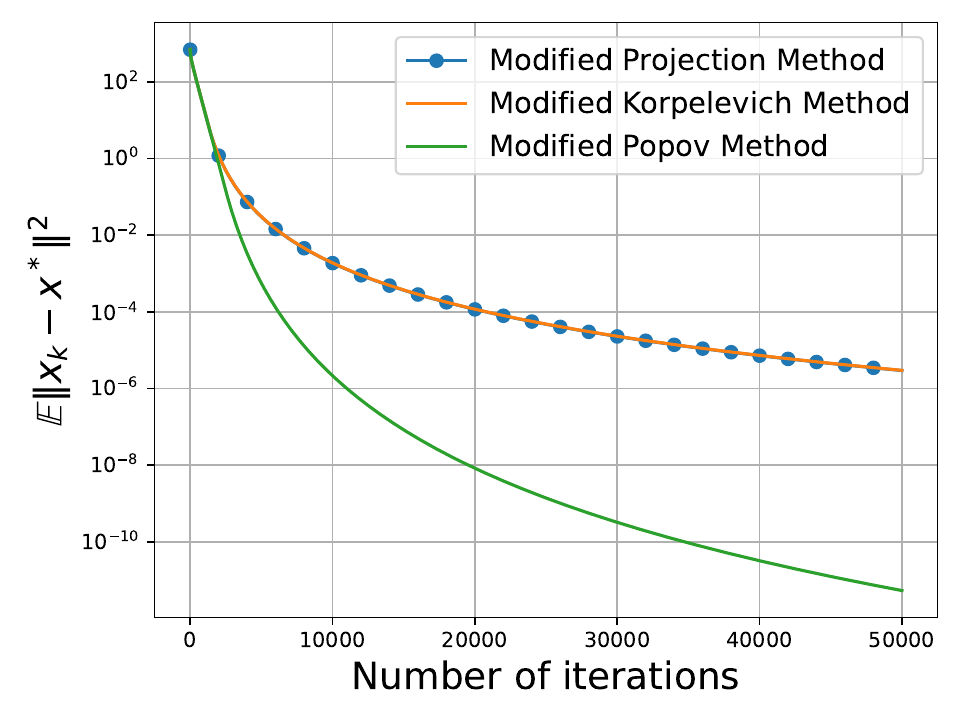}
		\caption{Plot~\ref{ill_cond2} with big $\a_0$}
		\label{fig:large_step}
	\end{subfigure}
\caption{Convergence results of two player matrix game for (a) $\kappa=3$, (b) $\kappa=20$, (c) $\kappa=1000$, and (d) $\kappa=1000$ 
for Algorithms~\ref{algo_proj_method}--\ref{algo_Popov_method} with $N_{k,j}=1, \forall j \in \cJ, k\geq 1$.}
\label{experiments}\vspace{-2pc} 
\end{figure*}


\textbf{Two player matrix game: } The game is as follows:
\begin{align}
    \text{Agent 1:} \quad & \min_{x_1 \in Y_1} \la x_1, A x_1 \ra + \la x_1, B x_2 \ra + \la p, x_1 \ra \nonumber\\
    &s.t. \; \la x_1, Q_i x_1 \ra + \la b_i, x_1 \ra \leq c_i , \; \forall i=1,\ldots,N, \nonumber\\
    \text{Agent 2:} \quad &\min_{x_2 \in Y_2} \la x_2, C x_2 \ra + \la x_1, D x_2 \ra + \la s, x_2 \ra \nonumber\\
    &s.t. \; \la x_2, R_i x_2 \ra + \la l_i, x_2 \ra \leq h_i , \; \forall i=1,\ldots,N , \nonumber
\end{align}
where the box constraints $Y_1 = Y_2 = [-10000,10000]^{100} \in \R^{100}$, and the joint decision vector across agents is $x = [x_1,x_2]^T \in \R^{200}$. The mapping $F$ and its Jacobian $\nabla F$ are
\begin{align}
    F(x) = \begin{bmatrix}
        A+A^T & B \\
        D^T & C+C^T
    \end{bmatrix} x + \begin{bmatrix}
        p \\
        s
    \end{bmatrix} , \quad \nabla F(x) = \begin{bmatrix}
        A+A^T & B \\
        D^T & C+C^T
    \end{bmatrix} . \label{mapping_F}
\end{align}
The matrices $A$, $B$, $C$, and $D$ are generated so that the Jacobian $\nabla F$ is a positive definite with eigenvalues in the interval $[\mu,L]$ for some $\mu>0$. To achieve that, we generate a diagonal matrix $\Lambda$ by sampling the eigenvalues from $U[\mu,L]$, where $U$ is the uniform distribution. Then, we create a matrix $M$, whose elements are sampled from a random normal distribution, and use QR decomposition to write it as $M = G K$, where $G$ is an orthogonal matrix 
$K$ is an upper triangular matrix. We use the matrix $G$ to set $\nabla F = G \Lambda G^T$, ensuring that the eigenvalues of $\nabla F$ lie within the range $[\mu, L]$. Once the matrix $\nabla F$ is generated, we compute the eigenvalues of $\nabla F$ and update $\mu$ to be its minimum eigenvalue and $L$ to be its maximum eigenvalue. 
The generated matrix $\nabla F$ is split into matrix blocks to obtain  $A+A^T$, $B$, $D^T$ and $C+C^T$. The vectors $p$ and $s$ are sampled from a random normal distribution. Now, for the generation of the constraints, we set the solution $x^*=[x_1^*,x_2^*]^T$ of the unconstrained VI$(\R^{100}\times \R^{100},F)$ to be feasible. 
Setting $F(x^*) = 0$, we get $x^* = [x_1^*,x_2^*]^T = - (\nabla F)^{-1} [p,s]^T$. The matrices $Q_i$ and $R_i$ are positive semi-definite with eigenvalues in the interval $[0,2]$; they are generated in the same way as $\nabla F$. The vectors $b_i$, and $l_i$ are generated from standard normal distribution, while $c_i$ and $h_i$ are given by $c_i = \la {x_1^*}, Q_i x_1^* \ra + \la b_i, x_1^* \ra + \delta_i$ and $h_i = \la {x_2^*}, R_i x_2^* \ra + \la l_i, x_2^* \ra + \chi_i$ for all $i = 1, \ldots, N$, where $\delta_i$ and $\chi_i$ are sampled uniformly from [1,2]. The constraint sets are convex since the matrices $Q_i$ and $R_i$ are positive semi-definite; hence Assumption \ref{asum_closed_set} is satisfied. Also, $Y_1$ and $Y_2$ are compact, implying by~\cite[Proposition~4.2.3]{bertsekas2003convex} that Assumption~\ref{asum-subgrad-bound} is satisfied. Moreover, the Jacobian of $F$ is a positive definite matrix, so Assumptions~\ref{asum_lipschitz} and \ref{asum_strong_monotone} are satisfied. The choice of $c_i$ and $h_i$ ensures that $g_{a_j}(x^*_j)<0$ for $j=1,2$ and $a_j \in \{1,\ldots, N\}$; thus, $x_1^*$ and $x_2^*$ are interior points of the respective constraint sets. Since all the functional constraints are quadratic, Assumption~\ref{asum-regularmod} is satisfied by 
\cite[Theorem 3.1]{luo1994extension}. 
The number of constraints is $N=10000$, with $\beta = 1$ and the initial iterate $x_0$ drawn from a random normal distribution. We ran $5$ experiments to compute the empirical average of the distance to the solution, which is presented in Figs.~\ref{well_cond}, \ref{ill_cond}, \ref{ill_cond2}, and \ref{fig:large_step}. 
We  set $N_{k,j}=1$ for all $k$ and $j=1,2$, where the single sample is drawn uniformly from the index set 
$\{1,\ldots,N\}$. Choosing a larger number of samples may not be beneficial  since the solution $x^*$ to VI$(S,F)$ is such that $F(x^*)=0$. We see that for well conditioned problem (Fig.~\ref{well_cond}), Algorithm \ref{algo_Popov_method} performs best. As the condition number $\kappa=L/\mu$ starts increasing (Fig.~\ref{ill_cond}), Algorithm~\ref{algo_Kor_method} gets better as the initial theoretical step size selection is larger than for Algorithms~\ref{algo_proj_method} and \ref{algo_Popov_method}. Due to this reason, Algorithms~\ref{algo_proj_method} and~\ref{algo_Popov_method} perform worse in Fig.~\ref{ill_cond2}, as they require longer time to converge, as they both have $O(\kappa^4/T^2)$ convergence rate for this scenario. For the last scenario, we have a simulation shown in Fig.~\ref{fig:large_step} with a larger initial step size $\a_0 = \frac{1}{4(L+\mu)}$ for Algorithms~\ref{algo_proj_method} and~\ref{algo_Popov_method}, which is same as Algorithm~\ref{algo_Kor_method}. Thus, the step sizes for Algorithms~\ref{algo_proj_method}, \ref{algo_Kor_method}, and~\ref{algo_Popov_method} are $\a_k = \min \left( \frac{2}{\mu(k+1)}, \frac{1}{4(L+\mu)} \right)$, $\a_k = \min \left( \frac{2}{\mu(k+1)}, \frac{1}{4(L+\mu)} \right)$, and $\a_k = \min \left( \frac{4}{\mu(k+1)}, \frac{1}{4(L+\mu)} \right)$ respectively. As indicated in the plots, all the methods converge a larger step size is beneficial for Algorithms~\ref{algo_proj_method} and~\ref{algo_Popov_method}, where Algorithm~\ref{algo_Popov_method} is converging faster. In Fig.~\ref{fig:large_step}, the initial step size of all the algorithms is not fine tuned and this simulation emphasises that all the algorithms do converge and a larger step size helped for worse conditioned problem.


\begin{figure*}[t]\centering
	\begin{subfigure}{0.48\textwidth}
		\includegraphics[width=\linewidth,height = 0.75\linewidth]
		{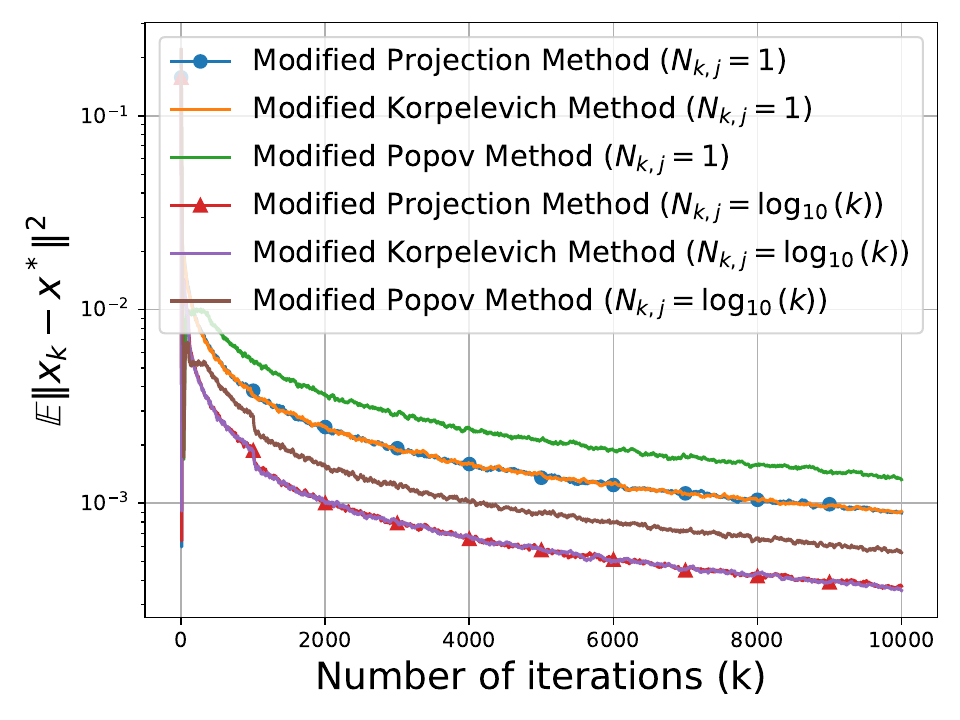}
		\caption{$\xi \sim U[0,0.1], \mu =1, L=3$}
		\label{imitation_error}
	\end{subfigure}
	\begin{subfigure}{0.48\textwidth}
		\includegraphics[width=\linewidth,height = 0.75\linewidth]
		{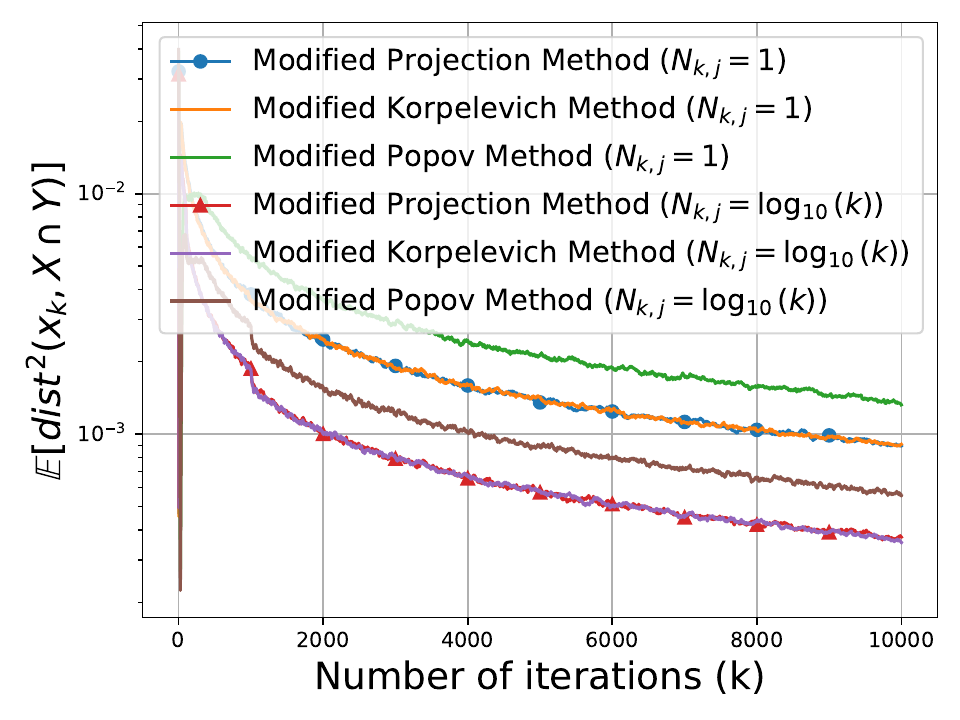}
		\caption{$\xi \sim U[0,0.1], \mu =1, L=3$}
		\label{imitation_dist}
	\end{subfigure}
\caption{Convergence results of two player imitation learning problem with $\kappa=3$ for Algorithms~\ref{algo_proj_method}--\ref{algo_Popov_method} plotted with both $N_{k,j}=1$ and $N_{k,j}= \lceil \log_{10}(k) \rceil$, $\forall j \in \cJ, k\geq 1$.}
\label{experiments2}\vspace{-1.5pc} 
\end{figure*}


\textbf{Imitation learning with exploration: } The game is as follows:
\begin{align}
    \text{Agent 1:} \quad &\min_{x_1 \in \R^2} \;\; \la x_1, x_2 \ra + \|x_1\|_2^2; \;\; s.t. \;\;  x_1\in[0.1,10]^2 ,  \nonumber\\
    \text{Agent 2:} \quad & \min_{x_2 \in \R^2} \;\; \la x_1, x_2 \ra + \|x_2\|_2^2; \;\; s.t. \;\; \|x_2-x_1\|_2^2 \leq \xi , \;\; \forall \xi \in [0,0.1],\label{imitation_prob}
\end{align}
where $\xi$ is a random exploration parameter with a uniform distribution on the interval [0,0.1], and the joint decision vector is $x = [x_1,x_2]^T \in \R^4$. The game can be viewed as a generalized game with the constraints of agent $2$ coupled with the constraints of agent $1$, and the set $Y = [0.1,10]^2 \times \R^2$.
Although the problem constraints don't fit exactly the framework of \eqref{eq-problem} (since agent $2$ has coupled constraints with the decision of agent $1$), still we can solve this problem using our approach where for each agent, we considering the other agent's decision fixed at the most recent iterate.  The mapping of the game is
$F(x) = \begin{bmatrix}
        2 & 1 \\
        1 & 2
    \end{bmatrix} x$ 
    The minimum and maximum eigenvalues of $\nabla F$ are $\mu = 1$ and $L=3$, so Assumptions~\ref{asum_lipschitz} and~\ref{asum_strong_monotone} are satisfied. 
    The set $X_1=\R^2$ while $X_2=\{x_1\}$ (a singleton), given that agent $1$ chooses $x_1$. Thus, the constraint sets $S_i=Y_i\cap X_i$
    of agents $i=1,2$ are closed, convex, and non-empty; hence Assumption~\ref{asum_closed_set} is satisfied. Also, since the constraint set of agent $1$ is compact and the decisions of agent $2$ lie within a ball with center $x_1$ and a maximal radius $\xi= 0.1$, the maximal constraint set of agent $2$ is also compact. Hence, Assumption~\ref{asum-subgrad-bound} is satisfied by~\cite[Proposition~4.2.3]{bertsekas2003convex}. 

    Next we show that $x^*= [0.1,0.1,0.1,0.1]^T$ is $x^*$ is the solution of VI$(S, F)$ associated with the game, i.e., $\la F(x^*), x-x^* \ra \geq 0, \forall x \in S$ with $x^* = [x_1^*,x_2^*]^T$. Since $\cap_{\xi \sim U[0,0.1]} \{x_2 : \|x_2-x_1\|_2^2 \leq \xi\} = \{x_1\}$, we have $x_1^* = x_2^*$, implying that  $F(x^*) = 3 [x_1^*,x_2^*]^T = 3x^*$. Thus, using the relation for the solution, we obtain $\la x^*, x\ra \geq \|x^*\|^2$ for all $x \in S$. Letting $x^*= [0.1,0.1,0.1,0.1]^T$ and $x =[x_1,x_2]^T$, the relation for the solution reduces to $\la x_1,{\bf 1}\ra +\la x_2,{\bf 1}\ra  \geq 0.4$, where ${\bf 1}$ is the vector of ones. 
    The constraint set for agent $1$ implies that $\la x_1,{\bf 1}\ra \geq 0.2$, and since the constraint set for agent to 2 is $x_2=x_1$, it follows that $\la x_1,{\bf 1}\ra +\la x_2,{\bf 1}\ra  \geq 0.4$,
    thus showing that $x^* = [0.1,0.1,0.1,0.1]^T$ is the solution of the VI associated with the game.
    
All the algorithms are initialized  with a random $x_0$ sampled from a uniform distribution on the interval $[0,1]$. For agent $1$, we directly project on its set $S_1=[0.1,10]^2$ (no random feasibility updates).
For agent $2$, we use Algorithm~\ref{algo_proj_steps} with $\beta = 1$
and 
take random realizations of $\xi \sim U[0,0.1]$ in each iteration with
two different experiment setups corresponding to sample sizes $N_{k,2}=1$ for all $k$ and $N_{k,2}=\lceil \log_{10}(k)\rceil$ for all $k$. Here, $\o_{k,2}^i$ in Algorithm~\ref{algo_proj_steps} is a random realization of $\xi$, drawn independently if the number of sampled constraints is larger than 1.
 We ran each experiment $1000$ times to compute the empirical average for squared distance of the iterates  to the solution and to the constraint set, plotted in Figs.~\ref{imitation_error} and~\ref{imitation_dist}, respectively. We see that Algorithms~\ref{algo_proj_method} and~\ref{algo_Kor_method} converge with similar rates and have better performance than Algorithm~\ref{algo_Popov_method}, due to a larger initial step~size $\a_0$. Moreover, the sample size $N_{k,2}=\lceil \log_{10}(k)\rceil$ works better than $N_{k,2}=1$ since $F(x^*)\ne0$ at the solution $x^*$ to the VI$(S,F)$, so a larger number of feasibility updates is beneficial.


\section{Conclusion and Future Directions}\label{sec:conclusion}
The paper presents the modified versions of Projection, Korpelevich, and Popov methods with random feasibility updates for solving the VI problem \eqref{VI_problem} with a
large (possibly infinite) number of constraints, represented as the intersection of convex functional level sets. We proved the convergence of the iterates in almost-sure sense and in-expectation, established the expected convergence rates, and simulated the algorithms for two problems (matrix game and imitation learning with exploration). 
Our analysis and results can be further extended to address merely monotone VIs, as well as stochastic VIs.

\bibliographystyle{siamplain}
\bibliography{references}
\end{document}